\documentclass[a4paper, 10pt, notitlepage]{article}

\usepackage{amsthm, amsmath, amssymb, latexsym}
\usepackage{mathrsfs,cite}
\usepackage[multiple]{footmisc}
\usepackage{graphicx}

\theoremstyle{plain}
\newtheorem{theorem}{Theorem}

\newtheorem{corollary}{Corollary}
\newtheorem{proposition}{Proposition}

\theoremstyle{definition}

\newtheorem{example}{Example}

\theoremstyle{remark}
\newtheorem{remark}{Remark}

\DeclareMathOperator{\co}{co}
\DeclareMathOperator{\cl}{cl}
\DeclareMathOperator{\dom}{dom}

\DeclareMathOperator*{\argmin}{arg\,min}

\DeclareMathOperator{\dist}{dist}

\author{Dolgopolik M.V.\footnote{Institute for Problems in Mechanical Engineering of the Russian Academy of Sciences,
Saint Petersburg, Russia}}
\title{New global optimality conditions for nonsmooth DC optimization problems}

\begin{document}

\maketitle

\begin{abstract}
In this article we propose a new approach to an analysis of DC optimization problems. This approach was largely
inspired by codifferential calculus and the method of codifferential descent, and is based on the use of a so-called
affine support set of a convex function instead of the Frenchel conjugate function. With the use of affine support sets
we define a global codifferential mapping of a DC function and derive new necessary and sufficient global optimality
conditions for DC optimization problems. We also provide new simple necessary and sufficient conditions for the global
exactness of the $\ell_1$ penalty function for DC optimization problems with equality and inequality constraints and
present a series of simple examples demonstrating a constructive nature of the new global optimality conditions. These
examples show that when the optimality conditions are not satisfied, they can be easily utilised in order to find
``global descent'' directions of both constrained and unconstrained problems. As an interesting theoretical example, we
apply our approach to the analysis of a nonsmooth problem of Bolza.
\end{abstract}

\section{Introduction}

For about thirty years DC optimization has been one of the most active areas of research in nonconvex optimization due
the abundance of applications and a possibility of the use of the well-developed apparatus of convex analysis and
convex optimization \cite{Tuy95_Collect,Tuy_book,HorstThoai,HorstPardalosThoai,Tuy_2000,LeThiDinh}. Various local
search 
\cite{Strekalovsky,BagUgon,TorBagKar,Strekalovsky2015,JokiBagirov,GaudiosoBagirov,JokiBagiov2,LeThiDinh,Polyakova2019}
and global search
\cite{Tuy95,Blanquero,Tuy2003,Strekalovsky2003,Ferrer,BigiFrangioni,BigiFrangioni2,StrekalovskyYanulevich}
methods for solving smooth and nonsmooth DC optimization problems were proposed over the years. It should be noted
that global search methods are often based on global optimality conditions, which have attracted a lot of attention of
researchers
\cite{Tuy87,Tuy87_Collect,HiriartUrruty_GlobalDC,HiriartUrruty_95,JeyakumarGlover,HiriartUrruty_98,DurHorstLocatelli,
Strekalovsky98,Tuy2003,Singer,DinhNghiaVallet,Polyakova,Zhang2013,Strekalovsky2017,Strekalovsky2017_2}.

The main goal of this article is to present new necessary and sufficient global optimality conditions for nonsmooth DC
optimization problems, including problems with DC equality and DC inequality constraints. These optimality conditions
were largely inspired by the codifferential calculus developed by professor V.F. Demyanov
\cite{Demyanov_InCollection_1988,Demyanov1988,Demyanov1989,DemRub_book} and are intimately connected to the method of
codifferential descent
\cite{DemRub_book,DemBagRub,BagUgon,TorBagKar,Dolgopolik_CodiffDescent,Dolgopolik_CodiffDescent_Global}. To obtain new
global optimality conditions, we introduce and study a so-called \textit{affine support set} of a proper closed convex
function. It should be noted that this set has been somewhat implicitly used in multiple monographs and papers on convex
analysis and optimization (see, e.g., \cite[Sect.~I.3]{EkelandTemam}, \cite[Theorem~1.3.8]{LemarechalII},
\cite[Sect.~7.3.3]{Rubinov}, etc.). However, to the best of author's knowledge, its properties have not been
thoroughly investigated earlier.

Affine support sets of convex functions play the same role in the non-positively homogeneous
case, as subdifferentials play in Minkowski duality. Furthermore, they are closely related to the abstract convexity
theory \cite{Rubinov} and Fenchel conjugate functions. In particular, almost all results on affine support sets
have natural counterparts in terms of Fenchel conjugate functions. However, the use of affine support sets provides one
with a new perspective on convex and DC functions, which allowed us to obtain a new result on convex functions
(part~\ref{Stat_Nonnegative} of Proposition~\ref{Thrm_MinViaGlobalHypodiff}). This result is a key ingredient in our
derivation of new global optimality conditions for DC optimization problems.

With the use of affine support sets we define a global codifferential mapping of a DC function, which can be viewed as
a ``globalization'' of Demyanov's definition of codifferential \cite{DemRub_book}. We provide some simple calculus rules
for global codifferentials that are particularly useful in the piecewise affine case. Furthermore, we utilise global
codifferentials and some results on affine support sets in order to obtain new necessary and sufficient global
optimality conditions for nonsmooth DC optimization problems in terms of global codifferentials (different global
optimality conditions in terms of codifferentials in the piecewise affine case were obtained by Polyakova
\cite{Polyakova}). It turns out that these condition are implicitly incorporated into the method of codifferential
descent (see Remark~\ref{Remark_GlobalOptimalityConditions} below and 
\cite{Dolgopolik_CodiffDescent,Dolgopolik_CodiffDescent_Global}) and have a somewhat constructive nature in the
piecewise affine case. Namely, we present a series of simple examples demonstrating that the verification of the global
optimality conditions at a non-optimal point allows one to find ``global descent'' directions, which sometimes lead
directly towards a global minimizer. In order to apply new global optimality conditions to problems with DC equality and
DC inequality constraints we obtain new simple necessary and sufficient conditions for the global exactness of the
$\ell_1$ penalty function. Finally, as an interesting theoretical example, in the end of the paper we apply some results
on global codifferentials of DC functions to an analysis of a nonsmooth problem of Bolza.

It should be noted that in many cases it is difficult to verify the global optimality conditions obtained in this
paper, since it is often difficult to compute a global codifferential of a DC function explicitly. However, a similar
statement is true for many other global optimality conditions for general DC optimization problems. Nevertheless, it
seems possible to design new numerical methods for DC optimization problem utilising a certain approximation of global
codifferential (cf. codifferential method in \cite{BagUgon}, aggregate codifferential method in \cite{TorBagKar}, and
Example~\ref{Example_NonPiecewiseAff_1} below).

The paper is organised as follows. In Section~\ref{Sect_AffineSupport} we introduce an affine support set of a convex
function, study its properties, and point out its connection with the Fenchel conjugate function.
Section~\ref{Sect_GlobOptCond} is devoted to necessary and sufficient global optimality conditions for nonsmooth DC
optimization problems in terms of global codifferentials. In this section, we also present a series of simple examples
demonstrating a somewhat constructive nature of the global optimality conditions and obtain simple conditions for the
global exactness of the $\ell_1$ penalty function for DC optimization problems with equality and inequality constrains.
Some connections of the global optimality conditions obtained in this paper with KKT optimality conditions and
global optimality conditions in terms of $\varepsilon$-subdifferentials are discussed in
Section~\ref{Sect_OtherOptimalityCond}. Finally, different global optimality conditions in terms of global
codifferentials and their application to an analysis of a nonsmooth problem of Bolza are given in
Section~\ref{Sect_Bolza}.

For the sake of simplicity, in this paper we study DC functions defined on a real Hilbert space. However, it should be
noted that most of the results of Sections~\ref{Sect_AffineSupport} and \ref{Sect_Bolza} (except for
part~\ref{Stat_Nonnegative} of Proposition~\ref{Thrm_MinViaGlobalHypodiff} and
Proposition~\ref{Theorem_MinOnConvSetViaHypodiff}) can be easily extended to the case of locally convex spaces, while
the
rest of the results of this paper (apart from Theorem~\ref{Thrm_GlobalExactPenFunc_DCProblem}) remain valid in strictly
convex reflexive Banach spaces.

\section{Affine support sets of convex functions}
\label{Sect_AffineSupport}

In this section we introduce and study a so-called affine support set of a closed convex function. The main ideas and
results presented below, in a sense, can be viewed as a natural extension of the Minkowski duality to the case of
general, i.e. non-positively homogeneous convex functions (cf. the abstract convexity theory in~\cite{Rubinov}). 

Let $\mathcal{H}$ be a real Hilbert space, $\overline{\mathbb{R}} = \mathbb{R} \cup \{ \pm \infty \}$, and 
$f \colon \mathcal{H} \to \overline{\mathbb{R}}$ be a proper closed convex function. As is well known 
(see, e.g. \cite[Prp.~I.3.1]{EkelandTemam}), the function $f$ can be represented as the supremum of a family of affine
functions. Taking, if necessary, the closed convex hull of this set, and identifying an affine function 
$l(x) = a + \langle v, x \rangle$ with the point $(a, v) \in \mathbb{R} \times \mathcal{H}$, one gets that there exists
a closed convex set $S_f \subset \mathbb{R} \times \mathcal{H}$ such that
$$
  f(x) = \sup_{(a, v) \in S_f} ( a + \langle v, x \rangle ) \quad \forall x \in \mathcal{H},
$$
where $\langle \cdot, \cdot \rangle$ is the inner product in $\mathcal{H}$. Any such set $S_f$ is called 
\textit{an affine support set} of the function $f$. At first, let us demonstrate how affine support sets are connected
with the $\varepsilon$-subdifferential of the function $f$.

\begin{proposition} \label{Thrm_SubdiffViaSupportSet}
For any affine support set $S_f$ of $f$ and for all  $\varepsilon \ge 0$ and $x \in \dom f$ one has
\begin{equation} \label{SubdiffViaSupportSet}
  \partial_{\varepsilon} f(x) = \big\{ v \in \mathcal{H} \bigm| 
  \exists a \in \mathbb{R} \colon (a, v) \in S_f, \: a + \langle v, x \rangle \ge f(x) - \varepsilon \big\}.
\end{equation}
\end{proposition}

\begin{proof}
Fix arbitrary $\varepsilon \ge 0$ and $x \in \dom f$, and denote by $D_{\varepsilon}(x)$ the set on the right-hand side
of \eqref{SubdiffViaSupportSet}. Observe that for any $(a, v) \in S_f$ such that 
$a + \langle v, x \rangle \ge f(x) - \varepsilon$ one has
$$
  f(y) - f(x) \ge a + \langle v, y \rangle - \big( a + \langle v, x \rangle \big) - \varepsilon 
  = \langle v, y - x \rangle - \varepsilon
  \quad \forall y \in \mathcal{H},
$$
which implies that $v \in \partial_{\varepsilon} f(x)$. 
Thus, $D_{\varepsilon}(x) \subseteq \partial_{\varepsilon} f(x)$. 

Arguing by reductio ad absurdum, suppose that $\partial_{\varepsilon} f(x) \ne D_{\varepsilon}(x)$. Then there exists
$v_0 \in \partial_{\varepsilon} f(x)$ such that $v_0 \notin D_{\varepsilon}(x)$. Hence $(a, v_0) \notin S_f$ for 
any $a \ge f(x) - \langle v_0, x \rangle - \varepsilon$, since otherwise $v_0 \in D_{\varepsilon}(x)$.

Denote $C_f = \{ (b, v) \in \mathbb{R} \times \mathcal{H} \mid \exists a \ge b \colon (a, v) \in S_f \}$. It
is clear that the set $C_f$ is convex, and $(f(x) - \langle v_0, x \rangle - \varepsilon, v_0) \notin C_f$. To apply 
the separation theorem, let us check that the set $C_f$ is closed. To this end, introduce a function 
$g \colon \mathcal{H} \to \mathbb{R}$ as follows: $g(v) = \sup\{ a \mid (a, v) \in S_f \}$. Observe that 
$(g(v), v) \in S_f$ for any $v \in \dom g$ due to the fact that the set $S_f$ is closed. Moreover, it is easy to see
that $C_f$ is the hypograph of the function $g$. Therefore, it is sufficient to check that the function $g$ is upper
semicontinuous (u.s.c.).

At first, note that $g$ is a proper concave function, since its hypograph is a convex set, and if $g(v) = + \infty$ for
some $v$ (i.e. $(a, v) \in S_f$ for any sufficiently large $a$), then $f(\cdot) \equiv + \infty$, which contradicts the
assumption that the function $f$ is proper. Note also that $g(\cdot) \not\equiv - \infty$, since otherwise 
$S_f = \emptyset$ and $f(\cdot) \equiv - \infty$, which contradicts our assumption. Furthermore, $g$ is bounded above
on any bounded set. Indeed, for any bounded set $Q \subset \mathcal{H}$ and $v \in Q$ either 
$(\mathbb{R} \times \{ v \}) \cap S_f = \emptyset$ and $g(v) = - \infty$ or $(a, v) \in S_f$ for some 
$a \in \mathbb{R}$, and
\begin{align*}
  g(v) &= \sup\{ a \mid (a, v) \in S_f \} 
  = \sup_{a \colon (a, v) \in S_f } \big( a + \langle v, x \rangle - \langle v, x \rangle \big) \\
  &\le \sup_{(a, v) \in S_f} \big( a + \langle v, x \rangle \big) - \langle v, x \rangle
  \le f(x) + q \| x \|,
\end{align*}
where $q = \sup_{v \in Q} \| v \|$ (recall that $x \in \dom f$, i.e. $f(x) < + \infty$).

Arguing by reductio ad absurdum suppose that $g$ is not u.s.c. at a point $v \in \mathcal{H}$. Let $v \in \dom g$. Then
there exists $\theta > 0$ such that for any $n \in \mathbb{N}$ one can find $v_n \in \dom g$ 
for which $g(v_n) > g(v) + \theta$ and $\| v_n - v \| < 1 / n$. Taking into account the fact that $g$ is bounded above
on bounded sets one gets that the sequence $\{ g(v_n) \}$ is bounded. Therefore, there exists a subsequence 
$\{ v_{n_k} \}$ such that the corresponding subsequence $\{ g(v_{n_k}) \}$ converges to some $g_* \ge g(v) + \theta$.
As was poited out above, $(g(v_{n_k}), v_{n_k}) \in S_f$ for all $k \in \mathbb{N}$. Hence passing to the limit as 
$k \to \infty$ and applying the closedness of the set $S_f$ one obtains that $(g_*, v) \in S_f$. Consequently, 
$g(v) \ge g_* \ge g(v) + \theta$, which is impossible.

Let now $v \notin \dom g$. Then there exist $M \in \mathbb{R}$ and a sequence $\{ v_n \} \subset \dom g$ converging to
$v$ such that $g(v_n) \ge M$ for all $n \in \mathbb{N}$. Applying, as above, the fact that the sequence 
$\{ g(v_n) \}$ is bounded one can extract a subsequence $\{ v_{n_k} \}$ such that the sequence $\{ g(v_{n_k}) \}$
converges to some $g_* \ge M > - \infty$. Therefore $(g_*, v) \in S_f$, and $g(v) \ge g_* > - \infty$, which is
impossible. Thus, $g$ is u.s.c., and the set $C_f$ is closed.

Recall that $(f(x) - \langle v_0, x \rangle - \varepsilon, v_0) \notin C_f$, and $C_f$ is a closed convex set. Applying
the separation theorem one obtains that there exist $(b, y) \in \mathbb{R} \times \mathcal{H}$ and $\delta > 0$ such
that
\begin{equation} \label{SubdiffVsSupport_Separation}
  b (f(x) - \langle v_0, x \rangle - \varepsilon) + \langle v_0, y \rangle \ge
  b a + \langle v, y \rangle + \delta \quad \forall (a, v) \in C_f.
\end{equation}
By definition for any $(a, v) \in S_f$ one has $(- \infty, a] \times \{ v \} \subset C_f$, which implies that 
$b \ge 0$. 

If $b > 0$, then dividing \eqref{SubdiffVsSupport_Separation} by $b$ and taking the supremum over all $(a, v) \in S_f$
one obtains
$$
  f(x) + \left\langle v_0, \frac{1}{b}y - x \right\rangle - \varepsilon 
  \ge f\left( \frac{1}{b}y \right) + \frac{\delta}{b}.
$$
Recall that $v_0 \in \partial_{\varepsilon} f(x)$. Therefore 
$$
  f\left( \frac{1}{b}y \right) \ge f(x) + \left\langle v_0, \frac{1}{b} y - x \right\rangle - \varepsilon
  \ge f\left( \frac{1}{b}y \right) + \frac{\delta}{b},
$$
which is impossible. Thus, $\partial_{\varepsilon} f(x) = D_{\varepsilon}(x)$.

Suppose now that $b = 0$. Then \eqref{SubdiffVsSupport_Separation} implies that
\begin{multline} \label{DirectDerivSeparationTh_Estimate}
  \frac{f(x + \alpha y) - f(x)}{\alpha} 
  = \frac{1}{\alpha} \Big( \sup_{(a, v) \in S_f} \big( a + \langle v, x + \alpha y \rangle \big) - f(x) \Big) \\
  \le \frac{1}{\alpha} \Big( \sup_{(a, v) \in S_f} (a + \langle v, x \rangle \rangle) 
  + \alpha \langle v_0, y \rangle - \alpha \delta - f(x) \Big) = \langle v_0, y \rangle - \delta
\end{multline}
for any $\alpha > 0$. On the other hand, by the definition of $\varepsilon$-subgradient for any 
$\alpha > \varepsilon / \delta$ one has
$$
  \frac{f(x + \alpha y) - f(x)}{\alpha} \ge \langle v_0, y \rangle - \frac{\varepsilon}{\alpha}
  > \langle v_0, y \rangle - \delta,
$$
which contradicts \eqref{DirectDerivSeparationTh_Estimate}. Thus, $\partial_{\varepsilon} f(x) =  D_{\varepsilon}(x)$,
and the proof is complete.	
\end{proof}

\begin{remark} \label{Remark_AffSupp_SupAttained}
By the proposition above the supremum in the definition of affine support set is attained for some $x \in \dom f$ iff
$f$ is subdifferentiable at $x$. In particular, if $f$ is finite-valued, then the supremum in the definition of affine
support set is attained for any $x \in \mathcal{H}$ by \cite[Proposition~I.5.2 and Corollary~I.2.5]{EkelandTemam}.
\end{remark}

Let $S_f$ be any affine support set of $f$. Our aim now is to show that several important properties of the function
$f$, such as boundedness below and the attainment of minimum, can be described in terms of simple geometric properties
of the set $S_f$. 

Observe that if $f$ attains a global minimum at a point $x_*$, then $0 \in \partial f(x_*)$, and 
$(f(x_*), 0) \in S_f$ by Proposition~\ref{Thrm_SubdiffViaSupportSet}. Thus, the sets $\mathbb{R} \times \{ 0 \}$ and
$S_f$ intersect. In the general case, define $a_f = \sup_{(a, 0) \in S_f} a$. By definition $a_f = - \infty$, if the
sets
$\mathbb{R} \times \{ 0 \}$ and $S_f$ do not intersect. Note also that if they do intersect, then $(a_f, 0) \in S_f$ due
to the facts that (i) this intersection is obviously closed, and (ii) if $a_f = + \infty$, then $f(\cdot) \equiv +
\infty$, which contradicts the assumption that the function $f$ is proper.

Denote by $N_f = \{ (b, w) \in \mathbb{R} \times \mathcal{H} \mid b( a - a_f ) + \langle w, v \rangle \le 0 \; 
\forall (a, v) \in S_f \}$ the normal cone to the set $S_f$ at the point $(a_f, 0)$, if the sets 
$\mathbb{R} \times \{ 0 \}$ and $S_f$ intersect, and define $N_f = \emptyset$ otherwise. From this point onwards we
suppose that the space $\mathbb{R} \times \mathcal{H}$ is endowed with the norm 
$\| (a, v) \| = \sqrt{a^2 + \| v \|^2}$.

\begin{proposition} \label{Thrm_MinViaGlobalHypodiff}
The following statements hold true:
\begin{enumerate}
\item{$f$ is bounded below iff $S_f \cap (\mathbb{R} \times \{ 0 \}) \ne \emptyset$;
\label{Stat_BoundedBelow}}

\item{if $f$ is bounded below, then $a_f = \inf_{x \in \mathcal{H}} f(x)$;
\label{Stat_InfValue}}

\item{$f$ attains a global minimum iff there exists $(b, w) \in N_f$ such that $b > 0$; furthermore,
$\argmin_{x \in \mathcal{H}} f(x) = \{ b^{-1} w \in \mathcal{H} \mid (b, w) \in N_f \colon b > 0 \}$;
\label{Stat_AttainmentOfMin}}

\item{if $f(x) \ge 0$ for all $x \in \mathcal{H}$, then either $0 \in S_f$ or $a_* > 0$, where $(a_*, v_*)$ is a
globally optimal solution of the problem
$$
  \min_{(a, v) \in \mathbb{R} \times \mathcal{H}} \| (a, v) \|^2 \quad 
  \text{subject to} \quad (a, v) \in S_f;
$$
conversely, if $f$ is bounded below and either $0 \in S_f$ or $a_* > 0$, then $f(x) \ge 0$ for all $x \in \mathcal{H}$.
Moreover, in the case $a_* > 0$ one has $a_f > 0$, i.e. $\inf_{x \in \mathcal{H}} f(x) > 0$.
\label{Stat_Nonnegative}}
\end{enumerate}
\end{proposition}

\begin{proof}
\ref{Stat_BoundedBelow}. If $S_f \cap (\mathbb{R} \times \{ 0 \}) \ne \emptyset$, then there exists 
$a_0 \in \mathbb{R}$ such that $(a_0, 0) \in S_f$. By the definition of $S_f$ for all $x \in \mathcal{H}$ one has 
$f(x) \ge a_0$, i.e. $f$ is bounded below.

Suppose, now, that $f$ is bounded below. Denote $f_* = \inf_{x \in \mathcal{H}} f(x)$. Then for any $\varepsilon > 0$
there exists $x_{\varepsilon} \in \mathcal{H}$ such that $f(x_{\varepsilon}) \le f_* + \varepsilon$. Hence
$0 \in \partial_{\varepsilon} f(x_{\varepsilon})$, which with the use of Proposition~\ref{Thrm_SubdiffViaSupportSet}
implies that there exists $a \ge f(x_{\varepsilon}) - \varepsilon$ such that $(a, 0) \in S_f$, i.e. 
$S_f \cap (\mathbb{R} \times \{ 0 \}) \ne \emptyset$. 

\ref{Stat_InfValue}. As was just proved, for any $\varepsilon > 0$ there exists 
$a \ge f(x_{\varepsilon}) - \varepsilon \ge f_* - \varepsilon$ such that $(a, 0) \in S_f$. Therefore 
$a_f \ge f_*$. On the other hand, for any $(a, 0) \in S_f$ and $x \in \mathcal{H}$ one obviously has 
$f(x) \ge a$, which implies that $a_f \le f_*$. Thus, $a_f = f_*$.

\ref{Stat_AttainmentOfMin}. Let $f$ attain a global minimum at a point $x_* \in \mathcal{H}$. By definition
$f(x_*) = \sup_{(a, v) \in S_f} (a + \langle v, x_* \rangle) = f_*$ or, equivalently,
$$
  (a - f_*) + \langle v, x_* \rangle \le 0 \quad \forall (a, v) \in S_f,
$$
which implies that $(1, x_*) \in N_f$ (note that $(f_*, 0) \in S_f$ and $a_f = f_*$ by the second part of the theorem).

Suppose, now, that $N_f \ne \emptyset$, and there exists $(b, w) \in N_f$ with $b > 0$. By the definition of $N_f$ and
the second part of the theorem one has
$$
  b (a - f_*) + \langle w, v \rangle \le 0 \quad \forall (a, v) \in S_f.
$$
Dividing by $b$ and taking the supremum over all $(a, v) \in S_f$ one obtains 
$$
  f\left( \frac{1}{b} w \right) = 
  \sup_{(a, v) \in S_f} \left( a + \left\langle v, \frac{1}{b} w \right\rangle \right) \le f_*,
$$
which implies that $b^{-1} w$ is a global minimizer of $f$. Thus,
$\argmin_{x \in \mathcal{H}} f(x) = \{ b^{-1} w \in \mathcal{H} \mid (b, w) \in N_f \colon b > 0 \}$.

\ref{Stat_Nonnegative}. Let $f(x) \ge 0$ for all $x \in \mathcal{H}$. Arguing by reductio ad absurdum,
suppose that $0 \notin S_f$ and $a_* \le 0$. From the definition of $(a_*, v_*)$ and the necessary condition for a
minimum of a differentiable function on a convex set it follows that
\begin{equation} \label{Nonneg_SepThrm}
  a_* (a - a_*) + \langle v_*, v - v_* \rangle \ge 0 \quad \forall (a, v) \in S_f.
\end{equation}
If $a_* = 0$, then one gets that $\langle v, -v_* \rangle \le - \| v_* \|^2 < 0$ for all $(a, v) \in S_f$ (note that
$v_* \ne 0$, since otherwise $0 \in S_f$). Therefore for any $\alpha > 0$ and $x \in \dom f$ one has
$$
  f(x - \alpha v_*) = \sup_{(a, v) \in S_f} (a + \langle v, x \rangle + \alpha \langle v, -v_* \rangle) \le 
  f(x) - \alpha \| v_* \|^2.
$$
Consequently, $f(x - \alpha v_*) \to - \infty$ as $\alpha \to + \infty$, which is impossible.

If $a_* < 0$, then dividing \eqref{Nonneg_SepThrm} by $a_*$ and taking the supremum over all $(a, v) \in S_f$ one
obtains that
$$
  f\left( \frac{1}{a_*} v_* \right) = 
  \sup_{(a, v) \in S_f} \left( a + \left\langle v, \frac{1}{a_*} v_* \right\rangle \right) 
  \le a_* + \frac{1}{a_*} \| v_* \|^2 < 0,
$$
which contradicts the assumption that $f$ is nonnegative.

Let us prove the converse statement. If $0 \in S_f$, then, obviously, one has $f(x) \ge 0$ for all 
$x \in \mathcal{H}$. Therefore, let $0 \notin S_f$ and $a_* > 0$. Arguing by reductio ad absurdum, suppose
that $f_* = \inf_{x \in \mathcal{H}} f(x) < 0$ (note that $f_* > - \infty$ due to the assumption that $f$ is bounded
below). By the second part of the theorem one has $(f_*, 0) \in S_f$. Consequently, for any $\alpha \in [0, 1]$ one has
$\alpha (a_*, v_*) + (1 - \alpha) (f_*, 0) \in S_f$. Setting $\alpha = |f_*| / (|f_*| + a_*) \in (0, 1)$ one obtains
that $(0, \alpha v_*) \in S_f$, which is impossible due to the definition of $(a_*, v_*)$ and the obvious inequality 
$\| (0, \alpha v_*) \|^2 < \| (a_*, v_*) \|^2$. Thus, the function $f$ is nonnegative. It remains to note that $a_f > 0$
in the case when $a_* > 0$ by virtue of the facts that $a_f \ge 0$ due to the nonnegativity of the function $f$, and
$a_f \ne 0$, since otherwise $0 \in S_f$ and $a_* = 0$.	
\end{proof}

\begin{remark} \label{Remark_MainThrm_GlobSuppSet}
{(i)~Let us note that the assumption on the boundedness below of the function $f$ cannot be dropped from the last part
of the proposition above. Indeed, if $f(x) \equiv a + \langle v, x \rangle$ with $a > 0$ and $v \ne 0$, then defining 
$S_f = (a, v)$ one obtains that $a_* > 0$, but the function $f$ is not nonnegative.
}

\noindent{(ii)~From the proof of the last part of the proposition above it follows that if $0 \notin S_f$, but $a_* =
0$,
then $f$ is not bounded below. Consequently, if $f$ is bounded below, then $f$ is nonnegative iff $a_* \ge 0$.
Furthermore, note that if $a_* < 0$, then $f(\frac{1}{a_*} v_*) < 0$.
}
\end{remark}

Let us give a simple example illustrating the proposition above.

\begin{example}
Let $\mathcal{H} = \mathbb{R}$, and $S_f = \{ (a, v) \in \mathbb{R}^2 \mid (a + 1)^2 + (v - 1)^2 \le 1 \}$. Then
according to Proposition~\ref{Thrm_MinViaGlobalHypodiff} one has $f_* = \inf_{x \in \mathbb{R}} f(x) = -1$. Furthermore,
it is easy to check that $N_f = \{ (a, v) \in \mathbb{R}^2 \mid a = 0, \: v \le 0 \}$, which 
by Proposition~\ref{Thrm_MinViaGlobalHypodiff} implies that the function $f$ does not attain a global minimum. Let us
verify this directly. Indeed, for any $x \in \mathbb{R}$ one has
$$
  f(x) = \max_{(a, v) \in S_f} (a + v x) = \max\{ (a - 1) + (v + 1) x \mid a^2 + v^2 \le 1 \} 
  = \sqrt{1 + x^2} + x - 1.
$$
Thus, $f_* = -1$, and $f$ does not attain a global minimum.
\end{example}

Let us also obtain an extension of part~\ref{Stat_Nonnegative} of Proposition~\ref{Thrm_MinViaGlobalHypodiff} to the
case when the nonnegativity of the function $f$ is checked on a set defined by an inequality constraint.

\begin{proposition} \label{Theorem_MinOnConvSetViaHypodiff}
Let $g \colon \mathcal{H} \to \overline{\mathbb{R}}$ be a proper closed convex function, and let $S_g$ be any
affine support set of $g$. Suppose also that $\dom f \cap \dom g \ne \emptyset$. If $f(x) \ge 0$ for all $x$ satisfying
the inequality $g(x) \le 0$, then either $0 \in \cl\co\{ S_f, S_g \}$ or $a_* > 0$, where $(a_*, v_*)$ is a globally
optimal solution of the problem
$$
  \min_{(a, v) \in \mathbb{R} \times \mathcal{H}} \| (a, v) \|^2 \quad 
  \text{subject to} \quad (a, v) \in \cl\co\{ S_f, S_g \}.
$$
Conversely, if $f$ is bounded below and continuous on the set $\{ x \mid g(x) \le 0 \}$, $0 \notin S_g$, and either 
$0 \in \cl\co\{ S_f, S_g \}$ or $a_* > 0$, then $f(x) \ge 0$ for all $x$ satisfying the inequality $g(x) \le 0$.
Moreover, in the case $a_* > 0$ there exists $\gamma > 0$ such that $f(x) \ge \gamma$ for all $x$ satisfying the
inequality $g(x) < \gamma$.
\end{proposition}

\begin{proof}
Let $f$ be nonnegative on the set $\{ x \mid g(x) \le 0 \}$. Arguing by reductio ad absurdum, suppose that 
$0 \notin \cl\co\{ S_f, S_g \}$ and $a_* \le 0$. By the necessary condition for a minimum of a convex function on
a convex set one obtains that
\begin{equation} \label{SepThrm_NonnegaInequalConstr}
  a_*(a - a_*) + \langle v_*, v - v_* \rangle \ge 0 \quad \forall (a, v) \in \cl\co\{ S_f, S_g \}.
\end{equation}
If $a_* < 0$, then dividing this inequality by $a_*$, and at first taking the supremum over all $(a, v) \in S_f$, and
at second taking the supremum over all $(a, v) \in S_g$ one obtains that
$$
  f\left( \frac{1}{a_*} v_* \right) \le a_* + \frac{1}{a_*} \| v_* \| < 0, \quad 
  g\left( \frac{1}{a_*} v_* \right) \le a_* + \frac{1}{a_*} \| v_* \| < 0,
$$
which is impossible. On the other hand, if $a_* = 0$, then from \eqref{SepThrm_NonnegaInequalConstr} it follows that
$\langle v, - v_* \rangle \le - \| v_* \|^2 < 0$ for all $(a, v) \in S_f \cup S_g$ (note that $v_* \ne 0$, since
otherwise $0 \in \cl\co\{ S_f, S_g \}$). Hence for any $x \in \dom f \cap \dom g$ and for all $\alpha > 0$ one has
$$
  f(x - \alpha v_*) \le f(x) - \alpha \| v_* \|^2, \quad 
  g(x - \alpha v_*) \le g(x) - \alpha \| v_* \|^2.
$$
Consequently, $f(x - \alpha v_*) < 0$ and $g(x - \alpha v_*) < 0$ for any sufficiently large $\alpha > 0$, which is
impossible.

Let us prove the converse statement. Define $h(\cdot) = \sup\{ f(\cdot), g(\cdot) \}$. It is easily seen that 
$\cl\co\{ S_f, S_g \}$ is an affine support set of the function $h$. Our aim is to verify that $f(x) \ge 0$ on
the set $\{ x \mid g(x) \le 0 \}$ iff $h(x) \ge 0$ for all $x \in \mathcal{H}$, provided $0 \notin S_g$. Then applying
the last part of Proposition~\ref{Thrm_MinViaGlobalHypodiff} to the function $h$ one obtains the desired result.

Clearly, if $f(x) \ge 0$ for all $x$ satisfying the inequality $g(x) \le 0$, then $h(\cdot) \ge 0$. Let us prove 
the converse statement. If the set $\{ x \mid g(x) \le 0 \}$ is empty, then the statement holds vacuously.
Therefore, suppose that this set is not empty. Note that if $\inf_{x \in \mathcal{H}} g(x) = 0$, then $0 \in S_g$ by
Proposition~\ref{Thrm_MinViaGlobalHypodiff}, which contradicts our assumption. Thus, there exists $x_0$ such that 
$g(x_0) < 0$, i.e. Slater's condition holds true.

Suppose that the function $h$ is nonnegative. Then $f(x) \ge 0$ for any $x$ satisfying the strict inequality $g(x) < 0$.
From the convexity of the function $g$ it follows that $\{ x \mid g(x) \le 0 \} = \cl\{ x \mid g(x) < 0 \}$, since for
any point $x$ such that $g(x) = 0$ one has $g(\alpha x + (1 - \alpha) x_0) < 0$ for all $\alpha \in [0, 1)$.
Consequently, applying the fact that $f$ is continuous on $\{ x \mid g(x) \le 0 \}$ one obtains that $f(x) \ge 0$ for
all $x$ satisfying the inequality $g(x) \le 0$, and the proof is complete.	
\end{proof}

\begin{remark}
The assumption that $f$ is bounded below on $\{ x \mid g(x) \le 0 \}$ is necessary for the validity of the converse
statement of the previous proposition. Indeed, if $f(x) = g(x) = a + \langle v, x \rangle$ for some $a > 0$ and 
$v \ne 0$, then $a_* = a > 0$, but $f(x) < 0$ for any $x$ such that $g(x) < 0$. The assumption $0 \notin S_g$ is also
necessary for the validity of the converse statement of the proposition, since if $0 \in S_g$, then 
$0 \in \cl\co\{ S_f, S_g \}$ regardless of the behaviour of the function $f$. Furthermore, note that the assumption 
$0 \notin S_g$ is, in fact, equivalent to  Slater's condition, provided the set $\{ x \mid g(x) \le 0 \}$ is not empty.
\end{remark}

With the use of Proposition~\ref{Thrm_MinViaGlobalHypodiff} we can point out a direct connection between affine support
sets of $f$ and the Frenchel conjugate function $f^*$.

\begin{proposition} \label{Theorem_AffSupport_vs_Conjugate}
Let $S_f$ be any affine support set of $f$. Then
\begin{equation} \label{ConjugateFunc_via_AffineSupport}
  \sup\{ a \mid (a, v) \in S_f \} = - f^*(v) \quad \forall v \in \mathcal{H}.
\end{equation}
In particular, any affine support set of the function $f$ is contained in the set 
$\{ (a, v) \in \mathbb{R} \times \mathcal{H} \mid a \le - f^*(v) \}$. Furthermore, the set
\begin{equation} \label{SmallestAffineSupportSet}
\begin{split}
  S_f &= \cl\co\{ (-f^*(v), v) \in \mathbb{R} \times \mathcal{H} \mid v \in \dom f^* \} \\
  &= \cl\co\{ (f(y) - \langle v, y \rangle, v) \in \mathbb{R} \times \mathcal{H} \mid 
  y \in \dom \partial f, \: v \in \partial f(y) \}
\end{split}
\end{equation}
is the smallest (by inclusion) affine support set of the function $f$.
\end{proposition}

\begin{proof}
Fix $v \in \mathcal{H}$, and consider the function $g(x) = f(x) - \langle v, x \rangle$. Note that this function is
bounded below iff $v \in \dom f^*$. On the other hand, from the fact that the set $S_f - (0, v)$ is an affine support
set of this function and the first part of Proposition~\ref{Thrm_MinViaGlobalHypodiff} it follows that $g$ is bounded
below iff there exists $a \in \mathbb{R}$ such that
$(a, v) \in S_f$. Furthermore, if $v \in \dom f^*$, then applying the second part of
Proposition~\ref{Thrm_MinViaGlobalHypodiff} one obtains that
\begin{align*}
  - f^*(v) = \inf_{x \in \mathcal{H}} (f(x) - \langle v, x \rangle) &= \sup\{ a \mid (a, 0) \in S_f - (0, v) \} \\
  &= \sup\{ a \mid (a, v) \in S_f \},
\end{align*}
i.e. \eqref{ConjugateFunc_via_AffineSupport} holds true, and 
$S_f \subseteq \{ (a, v) \in \mathbb{R} \times \mathcal{H} \mid a \le - f^*(v) \}$. Hence and from the fact that
\begin{equation} \label{ConvexFuncViaBiconjugate}
  f(x) = f^{**}(x) = \sup_{v \in \dom f^*} \big( \langle v, x \rangle - f^*(v) \big) 
  \quad \forall x \in \mathcal{H}
\end{equation}
it follows that set \eqref{SmallestAffineSupportSet} is the smallest affine support set of the function $f$. It
remains to note that the second equality in \eqref{SmallestAffineSupportSet} follows directly from the facts that 
(i) one can take the supremum in \eqref{ConvexFuncViaBiconjugate} over all $v \in \dom \partial f^*$ (since if 
$v \in \dom f^* \setminus \dom \partial f^*$, then for any $x \in \mathcal{H}$ by definition there exists 
$w \in \dom f^*$ such that  $\langle w, x \rangle - f^*(w) > \langle v, x \rangle - f^*(v)$), and 
(ii) $v \in \dom \partial f^*$ iff $v \in \partial f(y)$ for some $y \in \dom \partial f$ iff 
$f^*(v) = \langle v, y \rangle - f(y)$ by \cite[Corollary~X.1.4.4]{LemarechalII}. 
\end{proof}

\begin{remark}
The proposition above demonstrates that there is a direct connection between affine support sets and conjugate
functions.
Note, in particular, that the function $g(v)$ defined in the proof of Proposition~\ref{Thrm_SubdiffViaSupportSet} is, in
fact, the negative of the conjugate function $f^*$. Furthermore, Proposition~\ref{Thrm_SubdiffViaSupportSet} itself is a
reformulation of the standard characterization of $\varepsilon$-subgradients via the conjugate function (see,
e.g. \cite[Proposition~XI.1.2.1]{LemarechalII}) in terms of affine support sets. In the light of
Proposition~\ref{Theorem_AffSupport_vs_Conjugate} we can also give a simple interpretation of
Proposition~\ref{Thrm_MinViaGlobalHypodiff}. The first two statements of this proposition is nothing but the obvious
equality
$\inf_{x \in \mathcal{H}} f(x) = - f^*(0)$. The third one is a combination of the equality 
$\argmin_{x \in \mathcal{H}} f(x) = \partial f^*(0)$ and the well-known geometric interpretation of the
subdifferential in terms of the normal cone to the epigraph of a convex function 
(see, e.g. \cite[Proposition~VI.1.3.1]{Lemarechal}). However, to the best of author's knowledge, the last statement of
Proposition~\ref{Thrm_MinViaGlobalHypodiff} is completely new. Furthermore, the last statement of this propositon is a
basis of new global optimality conditions for DC optimization problems derived in the next section. 
\end{remark}

Let us present some simple calculus rules for affine support sets. Their proofs are straightforward and therefore
are omitted.

\begin{proposition}[linear combination]
Let $f_i \colon \mathcal{H} \to \overline{\mathbb{R}}$, $i \in I = \{ 1, \ldots, l \}$ be proper closed convex
functions, and let $S_{f_i}$ be any affine support set of $f_i$, $i \in I$. Then for any $\lambda_i \ge 0$, $i \in I$,
the set 
$S_f = \cl( \sum_{i \in I} \lambda_i S_{f_i} )$ is an affine support set of 
the function $f = \sum_{i \in I} \lambda_i f_i$. 
\end{proposition}

\begin{proposition}[affine transformation]
Let $g \colon \mathcal{H} \to \overline{\mathbb{R}}$ be a proper closed convex function, and $S_g$ be any affine support
set of $g$. Suppose also that $X$ is a Hilber space, $T \colon X \to \mathcal{H}$ is a bounded linear operator, and 
$f(x) = g(T x + b)$ for some $b \in \mathcal{H}$. Then the set 
\begin{equation} \label{ExactSupp_AffTrasform}
  S_f = \cl\big\{ (a + \langle v, b \rangle, T^* v) \in \mathbb{R} \times X \bigm| (a, v) \in S_g \big\}
\end{equation}
is an affine support set of $f$. Moreover, the closure operator in \eqref{ExactSupp_AffTrasform} can be dropped, if
$S_g$ is bounded or $T$ is invertible.
\end{proposition}

\begin{proposition}[supremum]
Let $Y$ be a nonempty set, and a function $f \colon \mathcal{H} \times Y \to \overline{\mathbb{R}}$ be
such that for any $y \in Y$ the function $f(\cdot, y)$ is proper, closed, and convex. Suppose also that $S(y)$ is an
affine support set of the function $f(\cdot, y)$, and the function $g(\cdot) = \sup_{y \in Y} f(\cdot, y)$ is proper.
Then $S_g = \cl\co\{ S(y) \mid y \in Y \}$ is an affine support set of the function $g$.
\end{proposition}

In the end of this section, let us give several simple examples demonstrating how one can compute affine support sets
of convex functions with the use of Proposition~\ref{Theorem_AffSupport_vs_Conjugate} and some other well-known results.

\begin{example}
If $f$ is a proper closed positively homogeneous convex function, then the set 
$S_f = \{ 0 \} \times \partial f(0)$ is an affine support set of $f$ (see, e.g.~\cite[Theorem~V.3.1.1]{Lemarechal}). 
\end{example}

\begin{example} \label{Example_SupposrtSet_PolyhedralFunc}
If $\mathcal{H} = \mathbb{R}^d$, and $f$ is a finite polyhedral convex function, then 
$f(x) = \max_{1 \le i \le n} (a_i + \langle v_i, x \rangle)$ for some $n \in \mathbb{N}$ and 
$(a_i, v_i) \in \mathbb{R}^{d + 1}$ (see~\cite[Sect.~19]{Rockafellar}). Consequently, 
the set $S_f = \co\{ (a_i, v_i) \mid i \in I \}$ is an affine support set of $f$. Therefore, a finite convex function
$f$ is polyhedral iff there exists an affine support set of this function that is a convex polytope.
\end{example}

\begin{example} \label{Example_SupportSetComputation}
If $f$ is G\^{a}teaux differentiable on its effective domain, then
$$
  S_f = \cl\co\Big\{ \big( f(x) - \langle \nabla f(x), x \rangle, \nabla f(x) \big) \in \mathbb{R} \times \mathcal{H}
  \Bigm| x \in \dom f \Big\}
$$
is an affine support set of $f$. Here $\nabla f(x)$ is the gradient of $f$ at $x$ In particular, 
if $f(x) = 0.5 \langle x, A x \rangle + \langle b, x \rangle$, where the linear operator 
$A \colon \mathcal{H} \to \mathcal{H}$ is positive semidefinite, then
$S_f = \cl\co\{ (- 0.5 \langle x, A x \rangle, Ax + b) \mid x \in \mathcal{H} \}$ is an affine support set of $f$. Note
that in this case it is easier to describe the affine support set with the use of the gradient rather than the
conjugate function (cf.~\eqref{SmallestAffineSupportSet}), since the conjugate function is defined via the pseudoinverse
operator of $A$.
\end{example}

\section{Global codifferential calculus and optimality conditions}
\label{Sect_GlobOptCond}

In this section we apply the main results on affine support sets of convex functions obtained above to DC
optimization problems. In particular, with the use of Proposition~\ref{Thrm_MinViaGlobalHypodiff} we obtain new
necessary and sufficient conditions for global optimality in DC optimization. Hereinafter we consider only
finite-valued DC functions $f \colon \mathcal{H} \to \mathbb{R}$ defined on the entire space $\mathcal{H}$.

Let $f$ be a DC function, i.e. let $f = g - h$, 
where $g, h\colon \mathcal{H} \to \mathbb{R}$ are closed convex functions. Suppose also that
$S_{g}$ and $S_{h}$ are any affine support sets of the functions $g$ and
$h$ respectively. Introduce the set-valued mappings
\begin{equation} \label{GlobalHypodiffDefViaSupportSet}
\begin{split}
  \underline{d} f(x) &= \big\{ (a - g(x) + \langle v, x \rangle, v) \in \mathbb{R} \times \mathcal{H} 
  | (a, v) \in S_{g} \big\}, \\
  \overline{d} f(x) &= \big\{ (- b + h(x) - \langle w, x \rangle, - w) \in \mathbb{R} \times \mathcal{H}
  | (b, w) \in S_{h} \big\}.
\end{split}
\end{equation}
Then for any $x, \Delta x \in \mathcal{H}$ the following equality holds true:
\begin{equation} \label{GlobalCodiffDef}
  f(x + \Delta x) - f(x) = \sup_{(a, v) \in \underline{d} f(x)} (a + \langle v, \Delta x \rangle)
  + \inf_{(b, w) \in \overline{d} f(x)} (b + \langle w, \Delta x \rangle)
\end{equation}
(in actuality, the supremum and the infimum are attained by Remark~\ref{Remark_AffSupp_SupAttained}). Indeed, by
definition one has
\begin{align} \notag
  g(x + \Delta x) - g(x) 
  &= \sup_{(a, v) \in S_{g}} (a + \langle v, x + \Delta x \rangle) - g(x) \\
  &= \sup_{(a, v) \in S_{g}} (a - g(x) + \langle v, x \rangle + \langle v, \Delta x \rangle) \notag \\
  &= \sup_{(a, v) \in \underline{d} f(x)} ( a + \langle v, \Delta x \rangle). \label{DC_ConvexPartViaGlobalCodiff}
\end{align}
Subtracting from this equality the same one for the function $h(x)$ one obtains that \eqref{GlobalCodiffDef}
is valid.
Furthermore, for any $x \in \mathcal{H}$ one has
$$
  \sup_{(a, v) \in \underline{d} f(x)} a = \sup_{(a, v) \in S_{g}} (a + \langle v, x \rangle) - g(x) = 0,
$$
and, similarly, $\inf_{(b, w) \in \overline{d} f(x)} b = 0$. Finally, observe that the sets $\underline{d} f(x)$ and
$\overline{d} f(x)$ are convex and closed due to the fact that the map
$(a, v) \mapsto (a - g(x) + \langle v, x \rangle, v)$ is a homeomorphism of $\mathbb{R} \times \mathcal{H}$. Thus,
the pair $[\underline{d} f(x), \overline{d} f(x)]$ has similar properties to codifferential of $f$ at $x$
\cite{DemRub_book,Demyanov_InCollection_1988,Demyanov1988,Demyanov1989}. Therefore, it is natural to call the pair
$D f = [\underline{d} f, \overline{d} f]$ \textit{a global codifferential mapping} (or simply \textit{a global
codifferential}) of the function $f$ associated with the DC decomposition $f = g - h$. The
multifunction $\underline{d} f$ is called \textit{a global hypodifferential} of $f$, while the multifunction
$\overline{d} f$ is called \textit{a global hyperdifferential} of $f$. Note that global codifferential mappings are
obviously not unique, since there exist infinitely many DC decompositions of a DC function.

Let us point out some simple calculus rules for global codifferentials. Their proofs are straightforward, and we omit
them for the sake of shortness (see~\cite[Proposition~4.4]{Dolgopolik_CodiffDescent_Global} for some details).

\begin{proposition} \label{Prp_GlobalCodiffCalc}
Let $f_i$, $i \in I = \{ 1, \ldots k \}$, be DC functions, and let $D f_i$ be the
global codifferential mapping of $f_i$ associated with a DC decomposition $f_i = g_{i} - h_{i}$.
The following statements hold true:
\begin{enumerate}
\item{if $f = f_1 + c$ for some $c \in \mathbb{R}$, then $D f = D f_1$;
}

\item{if $f = \sum_{i = 1}^k f_i$, then 
$D f = [ \cl(\sum_{i = 1}^k \underline{d} f_i), \cl(\sum_{i = 1}^k \overline{d} f_i) ]$
is a global codifferential mapping of the function $f$ associated with the DC decomposition 
$f = \sum_{i = 1}^k g_{i} - \sum_{i = 1}^k h_{i}$;
}

\item{if $f = \lambda f_1$, then $D f = [ \lambda \underline{d} f_1, \lambda \overline{d} f_1]$ is a global
codifferential mapping of $f$ associated with the DC decomposition 
$f = \lambda g_{1} - \lambda h_{1}$ in the case
$\lambda \ge 0$, and $D f = [ \lambda \overline{d} f_1, \lambda \underline{d} f_1]$ is a global codifferential 
mapping of $f$ associated with the DC decomposition $f = |\lambda| h_{1} - |\lambda| g_{1}$ in
the case $\lambda < 0$;
}

\item{if $f = \max_{i \in I} f_i$, then 
$$
  D f(\cdot) = \bigg[ \cl\co\bigg\{ (f_i(\cdot) - f(\cdot), 0) + \underline{d} f_i(\cdot)
  - \sum_{j \ne i} \overline{d} f_j(\cdot) \biggm| i \in I \bigg\}, 
  \cl\Big( \sum_{i = 1}^k \overline{d} f_i(\cdot) \Big)\bigg]
$$
is a global codifferential mapping of $f$ associated with the DC decomposition 
$f = \max_{i \in I}\{ g_{i} + \sum_{j \ne i} h_{j} \} - \sum_{i = 1}^k h_{i}$;
}

\item{if $f = \min_{i \in I} f_i$, then 
$$
  D f(\cdot) = \bigg[ \cl\Big( \sum_{i = 1}^k \underline{d} f_i(\cdot) \Big), 
  \cl\co\bigg\{ (f_i(\cdot) - f(\cdot), 0) + \overline{d} f_i(\cdot) 
  - \sum_{j \ne i} \underline{d} f_j(\cdot) \biggm| i \in I  \bigg\} \bigg]
$$
is a global codifferential mapping of $f$ associated with the DC decomposition 
$f = \sum_{i = 1}^k g_{i} - \max_{i \in I}\{ h_{i} + \sum_{j \ne i} g_{j} \}$.
}
\end{enumerate}
\end{proposition}

\begin{remark}
Let us explain the presence of the terms $(f_i(\cdot) - f(\cdot), 0)$ in the expressions for global codifferentials
of the functions $f = \max_{i \in I} f_i$ and $f = \min_{i \in I} f_i$ in the proposition above. The easiest way to see
this is by computing the increment of the function $f$. Namely, let $k = 2$ and $f = \max\{ f_1, f_2 \}$. Then 
for any $x, \Delta x \in \mathcal{H}$ one has
\begin{multline*}
  f(x + \Delta x) - f(x) = \max\{ f_1(x + \Delta x) - f(x), f_2(x + \Delta x) - f(x) \} \\
  = \max\Big\{ f_1(x) - f(x) + \sup_{(a, v) \in \underline{d} f_1(x)} (a + \langle v, \Delta x \rangle)
  + \inf_{(b, w) \in \overline{d} f_1(x)} (b + \langle w, \Delta x \rangle), \\
  f_2(x) - f(x) + \sup_{(a, v) \in \underline{d} f_2(x)} (a + \langle v, \Delta x \rangle)
  + \inf_{(b, w) \in \overline{d} f_2(x)} (b + \langle w, \Delta x \rangle) \Big\}.
\end{multline*}
Adding and subtracting the terms corresponding to $\overline{d} f_1(x)$ and $\overline{d} f_2(x)$ one obtains
\begin{multline*}
  f(x + \Delta x) - f(x) \\
  = \max\Big\{ f_1(x) - f(x) + \sup_{(a, v) \in \underline{d} f_1(x)} (a + \langle v, \Delta x \rangle) 
  - \inf_{(b, w) \in \overline{d} f_2(x)} (b + \langle w, \Delta x \rangle), \\
  f_2(x) - f(x) + \sup_{(a, v) \in \underline{d} f_2(x)} (a + \langle v, \Delta x \rangle) 
  - \inf_{(b, w) \in \overline{d} f_1(x)} (b + \langle w, \Delta x \rangle) \Big\} \\
  + \inf_{(b, w) \in \overline{d} f_1(x)} (b + \langle w, \Delta x \rangle)
  + \inf_{(b, w) \in \overline{d} f_2(x)} (b + \langle w, \Delta x \rangle).
\end{multline*}
which implies the required result. The interested reader can also verify this fact in a direct, but slightly more
complicated way. Namely, define 
$$
  g = \max\{ g_1 + h_2, g_2 + h_1 \}, \quad
  S_{g} = \cl\co\{ S_{g_1} + S_{h_2}, S_{g_2} + S_{h_1} \},
$$
and compute $\underline{d} f(x)$ with the use of 
\eqref{GlobalHypodiffDefViaSupportSet} 
(cf.~\cite[Proposition~4.4, part~(5)]{Dolgopolik_CodiffDescent_Global}).
\end{remark}

Now we can turn to the study of global optimality conditions for DC optimization problems. At first, we obtain
necessary and sufficient global optimality conditions for the unconstrained problem
$$
  \min_{x \in \mathcal{H}} f(x)	\eqno{(\mathcal{P}_0)}
$$
in terms of a global codifferential of the function $f$.

\begin{theorem} \label{Thrm_GlobOptCond}
Let $f$ be a DC function, $D f$ be any global codifferential of $f$, and $x_* \in \mathcal{H}$ be a given point.
Suppose that $f$ is bounded below, and a set $C \subseteq \overline{d} f(x_*)$
is such that $\overline{d} f(x_*) = \cl \co C$ (in particular, if $f = g - h$, then one can set 
$C = \{ (h^*(v) + h(x) - \langle v, x \rangle, -v) \mid v \in \dom h^* \}$). Then
$x_*$ is a globally optimal solution of the problem $(\mathcal{P}_0)$ if and only if for any $z \in C$ one has 
$a(z) \ge 0$, where $(a(z), v(z))$ is a globally optimal solution of the problem
$$
  \min_{(a, v) \in \mathbb{R} \times \mathcal{H}} \| (a, v) \|^2 \quad 
  \text{subject to} \quad (a, v) \in \underline{d} f(x_*) + z.
$$
\end{theorem}

\begin{proof}
From the definition of global codifferential mapping and the fact that $\overline{d} f(x_*) = \cl \co C$ it follows
that
$$
  f(x) - f(x_*) = \sup_{(a, v) \in \underline{d} f(x_*)} (a + \langle v, x - x_* \rangle)
  + \inf_{z \in C} (b + \langle w, x - x_* \rangle).
$$
Consequently, $x_*$ is a point of global minimum of $f$ iff for any $z \in C$ one has
$$
  \sup_{(a, v) \in \underline{d} f(x_*) + z} (a + \langle v, x - x_* \rangle) \ge 0 \quad \forall x \in \mathcal{H}.
$$
Note that the function on the left hand side of this inequality is bounded below by 
$\inf_{x \in \mathcal{H}} f(x) - f(x_*) > - \infty$. Hence applying the last part of
Proposition~\ref{Thrm_MinViaGlobalHypodiff} one obtains the desired result (see also
Remark~\ref{Remark_MainThrm_GlobSuppSet}).	
\end{proof}

\begin{corollary}
Let $f$ be a DC function, $D f$ be any global codifferential of $f$, and 
$x_* \in \mathcal{H}$ be a given point. Suppose that $f$ is bounded above, and a set $C \subseteq \underline{d} f(x_*)$
is such that $\underline{d} f(x_*) = \cl \co C$. Then $x_*$ is a point of global maximum of the function $f$ if and only
if for any $z \in C$ one has $b(z) \le 0$, where $(b(z), w(z))$ is a globally optimal solution of the problem
$$
  \min_{(b, w) \in \mathbb{R} \times \mathcal{H}} \| (b, w) \|^2 \quad 
  \text{subject to} \quad (b, w) \in \overline{d} f(x_*) + z.
$$
\end{corollary}

\begin{remark} \label{Remark_GlobalOptimalityConditions}
{(i)~From the proofs of Proposition~\ref{Thrm_MinViaGlobalHypodiff} and Theorem~\ref{Thrm_GlobOptCond} (see also
Remark~\ref{Remark_MainThrm_GlobSuppSet}) it follows that if $x_*$ is not a point of global minimum of the function
$f$, then there exists $z \in C$ such that $a(z) < 0$, and for any such $z \in C$ one has 
$f(x_* + a(z)^{-1}v(z)) < f(x_*)$. Thus, the necessary and sufficient global optimality conditions from the theorem
above not only allow one to verify whether a given point is a global minimizer, but also provide a way to compute
``better'' points, if the optimality conditions are not satisfied. Thus, it is fair to say that the global optimality
conditions in terms of global codifferentials are somewhat constructive. Furthermore, it seems possible to propose a
numerical method for general DC optimization problems based on the global optimality conditions from
Thereom~\ref{Thrm_GlobOptCond} and utilising a certain approximation of global codifferential
(cf.~\cite{BagUgon,TorBagKar}).
}

\noindent{(ii)~It should be noted that the global optimality conditions from Theorem~\ref{Thrm_GlobOptCond} (and
part~\ref{Stat_Nonnegative} of Proposition~\ref{Thrm_MinViaGlobalHypodiff}) were largely inspired by the codifferential
calculus and the method of codifferential descent proposed by Demyanov
\cite{DemRub_book,DemBagRub,Dolgopolik_CodiffDescent,Dolgopolik_CodiffDescent_Global}. As was pointed out in
\cite{Dolgopolik_AbstrConvApprox}, a nonsmooth function $f$ is codifferentiable iff its increment can be locally
approximated by a DC function. In the light of Theorem~\ref{Thrm_GlobOptCond} one can say that in every iteration of
the method of codifferential descent one verifies whether the global optimality conditions from
Theorem~\ref{Thrm_GlobOptCond} are satisfied, and then utilises ``global descent'' directions $v(z)$ of the DC
approximation as line search directions for the objective function (see \cite{Dolgopolik_CodiffDescent_Global} for
more details). Note that this observation partly explains the ability of the method of codifferential descent to ``jump
over'' some points of local minimum of the objective function (see~\cite{DemBagRub,Dolgopolik_CodiffDescent_Global} for
particular examples of this phenomenon).
}

\noindent{(iii)~It is obvious that in many particular cases the global optimality conditions from
Theorem~\ref{Thrm_GlobOptCond} are of theoretical value only, since it is extremely difficult to compute a global
codifferential of a DC function. However, the same statement is true for many other general global optimality
conditions. In particular, it is true for the well-known global optimality condition in terms of
$\varepsilon$-subdifferentials \cite{HiriartUrruty_GlobalDC,HiriartUrruty_95,HiriartUrruty_98} due to the fact that
$\varepsilon$-subdifferentials can be efficiently computed only in few particular cases (see, e.g. \cite{KumarLucet}).
Let us note that in the case when the function $f$ is piecewise affine, there always exists a global codifferential of
the function $f$ such that both sets $\underline{d} f(x)$ and $\overline{d} f(x)$ are convex polytopes
\cite{GorokhovikZorko}. In this case, a global
codifferential of the function $f$ can be computed with the aid of Proposition~\ref{Prp_GlobalCodiffCalc}. See
\cite{Dolgopolik_CodiffDescent_Global} for applications of the optimality conditions from the theorem above to design
and analysis of numerical methods for global optimization of nonconvex piecewise affine functions.
}
\end{remark}

Let us give a simple example illustrating the use of the global optimality conditions from
Theorem~\ref{Thrm_GlobOptCond}.

\begin{example}
Let $\mathcal{H} = \mathbb{R}$, $f(x) = \min\{ 2|x|, |x + 2| + 1 \}$, and $x_0 = -2$. Let us check the optimality
conditions at $x_0$. Note that $x_0$ is a point of strict local minimum of the function $f$, while a global minimum is
attained at the point $x_* = 0$.

Denote $f_1(x) = 2|x|$ and $f_2(x) = |x + 2| + 1$. With the use of Proposition~\ref{Prp_GlobalCodiffCalc} one gets that
\begin{gather*}
  \underline{d} f_1(x_0) = \co\left\{ \begin{pmatrix} 0 \\ -2 \end{pmatrix}, \begin{pmatrix} -8 \\ 2 \end{pmatrix}
  \right\}, \quad \overline{d} f_1(x_0) = \{ 0 \}, \\
  \underline{d} f_2(x_0) = \co\left\{ \begin{pmatrix} 0 \\ 1 \end{pmatrix}, \begin{pmatrix} 0 \\ -1 \end{pmatrix}
  \right\}, \quad \overline{d} f_2(x_0) = \{ 0 \}
\end{gather*}
(here the first coordinate is $a$, and the second one is $v$). Observe that unlike all subdifferentials, a
codifferential is a pair of \textit{two}-dimensional convex sets even in the one-dimensional case. Applying
Proposition~\ref{Prp_GlobalCodiffCalc} again one obtains that
$$
  \underline{d} f(x_0) = \underline{d} f_1(x_0) + \underline{d} f_2(x_0) =
  \co\left\{ \begin{pmatrix} 0 \\ -1 \end{pmatrix}, \begin{pmatrix} 0 \\ -3 \end{pmatrix},
  \begin{pmatrix} -8 \\ 3 \end{pmatrix}, \begin{pmatrix} -8 \\ 1 \end{pmatrix}
  \right\},
$$
and
$$
  \overline{d} f(x_0) = \co\left\{ \begin{pmatrix} 3 \\ 0 \end{pmatrix} - \underline{d} f_2(x_0), 
  - \underline{d} f_1(x_0) \right\}
  = \co\left\{ \begin{pmatrix} 3 \\ 1 \end{pmatrix}, \begin{pmatrix} 3 \\ -1 \end{pmatrix},
  \begin{pmatrix} 0 \\ 2 \end{pmatrix}, \begin{pmatrix} 8 \\ -2 \end{pmatrix}
  \right\}.
$$
Let $C$ be the set of extreme points of $\overline{d} f(x_0)$, i.e.
$$
  C = \left\{ \begin{pmatrix} 3 \\ 1 \end{pmatrix}, \begin{pmatrix} 3 \\ -1 \end{pmatrix},
  \begin{pmatrix} 0 \\ 2 \end{pmatrix}, \begin{pmatrix} 8 \\ -2 \end{pmatrix}
  \right\}.
$$
Then one can easily verify that
\begin{enumerate}
\item{$0 \in \underline{d} f(x_0) + z$ for $z = (3, 1) \in C$, $z = (0, 2) \in C$, and $z = (8, -2) \in C$;
}

\item{$(a(z), v(z)) = (-0.2, -0.4)$ for $z = (3, -1) \in C$.
}
\end{enumerate}
Thus, the global optimality conditions from Theorem~\ref{Thrm_GlobOptCond} are not satisfied. Furthermore, note that
for $z = (3, -1)$ one has $x(z) = x_0 + {a(z)}^{-1} v(z) = 0$, i.e. $x(z)$ is a point of global minimum of the function
$f$.
\end{example}

Now we turn to constrained DC optimization problems. We start with the case of inequality constrained problems, since
the presence of equality constraints significantly complicates the derivation of optimality conditions. Below, we
largely follow Proposition~\ref{Theorem_MinOnConvSetViaHypodiff}, but do not apply it directly, since, as one can
verify, a direct application of this theorem leads to more restrictive regularity assumptions on the constraints.

Consider the optimization problem of the form
$$
  \min_{x \in \mathcal{H}} f_0(x) \quad \text{subject to} \quad f_i(x) \le 0, \quad i \in I,
  \eqno{(\mathcal{P}_I)}
$$
where $f_i$, $i \in 0 \cup I$, $I = \{ 1, \ldots, l \}$, are DC functions. Denote by
$\Omega$ the feasible region of this problem. To obtain global optimality conditions for this problem we need to impose
a regularity assumption on the constraints. Namely, one says that \textit{the interior point constraint qualification}
(IPCQ) holds at a point $x_0 \in \Omega$, if $x_0 \in \cl \{ x \in \mathcal{H} \mid f_i(x) < 0 \: i \in I \}$ or,
equivalently, if for any $\varepsilon > 0$ there exists $y \in \Omega$ such that $\| y - x \| < \varepsilon$, and
$f_i(y) < 0$ for all $i \in I$. It is easy to see that in the case when the functions $f_i$, $i \in I$, are convex, 
IPCQ is equivalent to Slater's condition. Note also that IPCQ holds at $x_0$, in particular, if a nonsmooth
Mangasarian-Fromovitz constraint qualification (MFCQ) holds true at this point, i.e. if there exists 
$v \in \mathcal{H}$ such that $f_i'(x_0, v) < 0$ for all $i \in I$ such that $f_i(x_0) = 0$, where $f_i'(x_0, v)$ is the
directional derivative of $f_i$ at $x_0$ in the direction $v$. Finally, it should be noted that IPCQ is a
generalization of the robustness condition from \cite{HorstThoai}.

\begin{theorem} \label{Theorem_GlobOptCond_InequalConst}
Let there exist a globally optimal solution of the problem $(\mathcal{P}_I)$  such that IPCQ holds true at this
solution, and let $x_*$ be a feasible point of  $(\mathcal{P}_I)$. Let also the function $f_0$ be
bounded below on $\Omega$, and $D f_i$ be a global codifferential of $f_i$, $i \in I \cup \{ 0 \}$. Suppose, finally,
that $C_i \subseteq \overline{d} f_i(x_*)$ is a nonempty set such that $\overline{d} f_i(x_*) = \cl\co C_i$, 
$i \in I \cup \{ 0 \}$. Then $x_*$ is a globally optimal solution of the problem $(\mathcal{P}_I)$ if and only
if for any $z_i \in C_i$, $i \in I \cup \{ 0 \}$, one has $a(z) \ge 0$, where $(a(z), v(z))$ with
$z = (z_0, z_1, \ldots, z_l)$ is a globally optimal solution of the problem
$$
  \min_{(a, v) \in \mathbb{R} \times \mathcal{H}} \| (a, v) \|^2 \quad 
  \text{subject to} \quad (a, v) \in L(z)
$$
and 
\begin{equation} \label{TrickPenaltyInequal}
  L(z) = \cl\co\{ \underline{d} f_0(x_*) + z_0, \: \underline{d} f_i(x_*) + z_i + (f_i(x_*), 0) \mid i \in I \}.
\end{equation}
\end{theorem}

\begin{proof}
Let us utilise a global version of the standard trick (see, e.g. the classic paper \cite{Ioffe_Penalty}) to transform
the problem $(\mathcal{P}_I)$ into an unconstrained optimization problem. Introduce the function
$$
  F(x) = \max\{ f_0(x) - f_0(x_*), f_1(x), \ldots, f_l(x) \}.
$$
Note that $F(x_*) = 0$, since $x_*$ is a feasible point of 
$(\mathcal{P}_I)$. Let us check that $x_*$ is a globally optimal solution of the problem $(\mathcal{P}_I)$ iff it is a
point of global minimum of the function $F$. 

Indeed, suppose that $x_*$ is a globally optimal solution of the problem $(\mathcal{P}_I)$. Observe that if
$F(x) < 0$ for some $x \in \mathcal{H}$, then $x \in \Omega$ and $f_0(x) < f_0(x_*)$, which is impossible. Thus, 
$F(x) \ge F(x_*) = 0$ for any $x \in \mathcal{H}$, i.e. $x_*$ is a point of global minimum of the function $F$.
Conversely, let $x_*$ be a point of global minimum of $F$. By definition $F(x) \ge F(x_*) = 0$ for all 
$x \in \mathcal{H}$. Hence, in particular, for any $x$ such that $f_i(x) < 0$ for all $i \in I$ one has 
$f_0(x) \ge f_0(x_*)$. Thus, $x_*$ is a globally optimal solution of the problem
$$
  \min_{x \in \mathcal{H}} f_0(x) \quad \text{subject to} \quad 
  x \in \{ x_* \} \cup \{ y \in \mathcal{H} \mid f_i(y) < 0 \: \forall i \in I \}.
$$
Note that the function $f_0$ is continuous as the difference of finite closed convex functions that are continuous
due to the fact that $\mathcal{H}$ is a Hilbert space (see, e.g. \cite[Corollary~I.2.5]{EkelandTemam}). Therefore,
taking into account the fact that by our assumption IPCQ holds true at some globally optimal solution $y_*$ of the
problem $(\mathcal{P}_I)$ one obtains that $f_0(x_*) \le f_0(y_*)$, which implies that $x_*$ is a globally optimal
solution of $(\mathcal{P}_I)$ as well. Thus, $x_*$ is a globally optimal solution of the 
problem $(\mathcal{P}_I)$ iff $x_*$ is a point of global minimum of the function $F$.

From the definition of global codifferential it follows that
\begin{multline*}
  F(x) = \max_{i \in I} 
  \Big\{ \sup_{(a, v) \in \underline{d} f_0(x_*)} (a + \langle v, x - x_* \rangle) +
  \inf_{(b, w) \in \overline{d} f_0(x_*)} (b + \langle w, x - x_* \rangle), \\
  f_i(x_*) + \sup_{(a, v) \in \underline{d} f_i(x_*)} (a + \langle v, x - x_* \rangle) +
  \inf_{(b, w) \in \overline{d} f_i(x_*)} (b + \langle w, x - x_* \rangle)
  \Big\}.
\end{multline*}
Therefore, as is easy to see, $x_*$ is a point of global minimum of the function $F$ iff for any 
$z_i \in C_i$, $i \in I \cup \{ 0 \}$ the function
\begin{align*}
  F_z(x) = \max_{i \in I} 
  \Big\{ &\sup_{(a, v) \in \underline{d} f_0(x_*) + z_0} (a + \langle v, x - x_* \rangle), \\
  &f_i(x_*) + \sup_{(a, v) \in \underline{d} f_i(x_*) + z_i} (a + \langle v, x - x_* \rangle)
  \Big\}
\end{align*}
is nonnegative. Note that 
$F_z(x) \ge F(x) \ge f_0(x) - f_0(x_*) \ge \inf_{x \in \Omega} f_0(x) - f_0(x_*) > - \infty$ for any 
$x \in \Omega$, and $F_z(x) \ge F(x) > 0$ for any $x \notin \Omega$, i.e. the function $F_z$ is bounded below.
Consequently, taking into account the fact that the set \eqref{TrickPenaltyInequal} is an affine support set of
$F_z$, and applying the last part of Proposition~\ref{Thrm_MinViaGlobalHypodiff} one obtains the desired result.

\end{proof}

\begin{remark} \label{Remark_InteriorPointStep}
{(i)~As in the case of Theorem~\ref{Thrm_GlobOptCond}, the global optimality conditions from the theorem above are
somewhat constructive. Namely, one can easily verify that if $x_*$ is not a globally optimal solution of the problem
$(\mathcal{P}_I)$, then for any $z_i \in C_i$, $i \in I \cup \{ 0 \}$ such that $a(z) < 0$ (note that such
$z_i$ exist by Theorem~\ref{Theorem_GlobOptCond_InequalConst}) one has 
$f_0( x_* + a(z)^{-1} v(z) ) < f_0(x_*)$ and $f_i( x_* + a(z)^{-1} v(z) ) < 0$ for all $i \in I$. Thus, if $x_*$ is not
a globally optimal solution of the problem $(\mathcal{P}_I)$, then with the use of the global optimality conditions
from Theorem~\ref{Theorem_GlobOptCond_InequalConst} one can find a ``better'' point in the interior of the feasible
region (see Example~\ref{Example_InequalConstrProblem} below).
}

\noindent{(ii)~Note that if $x_*$ is a point of global minimum of the function $F(x)$ defined in the proof of the
theorem above, but IPCQ does not hold true at any globally optimal solution of the problem $(\mathcal{P}_I)$, then
$x_*$ need not be a globally optimal solution of this problem. For example, if $l = 2$, $f_1(x) = \| x \| - 1$, 
$f_2(x) = 1 - \| x \|$, then IPCQ does not hold true at any feasible point of $(\mathcal{P}_I)$, and
any feasible point $x_*$ is a global minimizer of $F(x)$. Thus, the validity of IPCQ is, in essence, necessary for the
validity of the global optimality conditions from the theorem above. Furthermore, this example shows that
Theorem~\ref{Theorem_GlobOptCond_InequalConst} cannot be applied to equality constrained problems, since IPCQ fails to
hold true, if one rewrites an equality constraint $f_i(x) = 0$ as two inequality constraints $f_i(x) \le 0$ and 
$- f_i(x) \le 0$.
}
\end{remark}

Let us give a simple example illustrating Theorem~\ref{Theorem_GlobOptCond_InequalConst}.

\begin{example} \label{Example_InequalConstrProblem}
Let $\mathcal{H} = \mathbb{R}$, and the problem $(\mathcal{P}_I)$ have the form
\begin{equation} \label{ExProblem_InequalConstr}
  \min_{x \in \mathbb{R}} f_0(x) = |x - 4| \quad \text{subject to} \quad
  f_1(x) = \min\{ |x - 2|, |x + 2| \} - 1 \le 0.
\end{equation}
Let also $x_0 = -1$. It is easily seen that $\Omega = [-3, -1] \cup [1, 3]$, IPCQ holds true at the unique globally
optimal solution $x_* = 3$ of problem \eqref{ExProblem_InequalConstr}, and $x_0$ is a locally optimal solution of
this problem.  Let us check the global optimality conditions at the point $x_0$.

With the use of Proposition~\ref{Prp_GlobalCodiffCalc} one obtains that
\begin{align*}
  \underline{d} f_0(x_0) &= \co\left\{ \begin{pmatrix} -10 \\ 1 \end{pmatrix}, 
  \begin{pmatrix} 0 \\ -1 \end{pmatrix} \right\}, \quad \overline{d} f_0(x_0) = \{ 0 \}, \\
  \underline{d} f_1(x_0) &= \co\left\{ \begin{pmatrix} -6 \\ 2 \end{pmatrix}, \begin{pmatrix} -8 \\ 0 \end{pmatrix},
  \begin{pmatrix} 0 \\ 0 \end{pmatrix}, \begin{pmatrix} -2 \\ -2 \end{pmatrix} \right\}, \\
  \overline{d} f_1(x_0) &= \co\left\{ \begin{pmatrix} 2 \\ -1 \end{pmatrix}, \begin{pmatrix} 4 \\ 1 \end{pmatrix},
  \begin{pmatrix} 6 \\ -1 \end{pmatrix}, \begin{pmatrix} 0 \\ 1 \end{pmatrix} \right\}.
\end{align*}
Let $C_0 = \{ 0 \}$, and $C_1$ be the set of extreme points of $\overline{d} f_1(x_0)$. Then applying
Theorem~\ref{Theorem_GlobOptCond_InequalConst} one can check that
\begin{enumerate}
\item{$0 \in L(z)$ for $z = (z_0, z_1)$ with $z_0 = (0, 0) \in C_0$, $z_1 = (4, 1) \in C_1$, 
$z_1 = (6, -1) \in C_1$, and $z_1 = (0, 1) \in C_1$;
}

\item{$(a(z), v(z)) = (-0.1, -0.3)$ for $z = (z_0, z_1)$ with $z_0 = (0, 0) \in C_0$ and 
$z_1 = (2, -1) \in C_1$.
}
\end{enumerate}
Thus, the global optimality conditions from Theorem~\ref{Theorem_GlobOptCond_InequalConst} are not satisfied.
Furthermore, note that in the case $z_1 = (2, -1)$ one has $x_1 = x_0 + a(z)^{-1} v(z) = 2$, 
$f_0(x_1) = 2 < 5 = f_0(x_0)$ and $f_1(x_1) = -1 < 0$.
\end{example}

Now we turn to the general constrained optimization problem of the form
$$
  \min_{x \in \mathcal{H}} f_0(x) \quad 
  \text{s.t.} \quad f_i(x) \le 0, \quad i \in I, \quad f_j(x) = 0, \quad j \in J,
  \eqno{(\mathcal{P}_{IJ})}
$$
where $f_i$, $i \in 0 \cup I \cup J$, $I = \{ 1, \ldots, l \}$, 
$J = \{ l + 1, \ldots, m \}$ are DC functions. Denote by $\Omega$ the feasible region of the problem 
$(\mathcal{P}_{IJ})$, and introduce the function
$$
  \varphi(x) = \sum_{i = 1}^l \max\{ 0, f_i(x) \} + \sum_{j = l + 1}^m |f_j(x)|.
$$
Observe that $\Omega = \{ x \in \mathcal{H} \mid \varphi(x) = 0 \}$.

Our aim is to provide simple sufficient conditions under which the merit function 
$F_{\lambda}(x) = f_0(x) + \lambda \varphi(x)$ for the problem $(\mathcal{P}_{IJ})$ is globally exact in the finite
dimensional case (note that this function is DC, if the problem $(\mathcal{P}_{IJ})$ is DC). Apart from its
direct applications to the design of numerical methods for solving the problem $(\mathcal{P}_{IJ})$, this result can
also be used for the derivation of global optimality conditions for the problem $(\mathcal{P}_{IJ})$. 

Recall that the function $F_{\lambda}$ is said to be (globally) exact for the problem $(\mathcal{P}_{IJ})$, if there
exists $\lambda_* \ge 0$ such that for any $\lambda \ge \lambda_*$ the set of globally optimal solutions of the problem
$(\mathcal{P}_{IJ})$ coincides with the set of global minimizers of the function $F_{\lambda}$, i.e. the problem
$(\mathcal{P}_{IJ})$ is equivalent (in terms of globally optimal solutions) to the penalized problem
$$
  \min_{x \in \mathcal{H}} F_{\lambda}(x)	\eqno{(\mathcal{P}_{\lambda})}
$$
for any $\lambda \ge \lambda_*$. The greatest lower bound
of all such $\lambda_*$ is called \textit{the least exact penalty parameter} of the function $F_{\lambda}$.

\begin{theorem} \label{Thrm_GlobalExactPenFunc_DCProblem}
Let $\mathcal{H}$ be finite dimensional. Suppose that $\varphi$ has a local error bound at every globally optimal
solution of the problem $(\mathcal{P}_{IJ})$, i.e. for any globally optimal solution $x_*$ of this problem 
there exist $\tau > 0$ and a neighbourhood $U$ of $x_*$ such that
\begin{equation} \label{PenTermErrorBound}
  \varphi(x) \ge \tau \dist(x, \Omega)  \quad \forall x \in U.
\end{equation}
Then the function $F_{\lambda}$ is globally exact if and only if there exists $\lambda \ge 0$ such that the set
$\{ x \in \mathcal{H} \mid F_{\lambda}(x) < f_* \}$ is either bounded or empty, where $f_*$ is the optimal value of
the problem $(\mathcal{P}_{IJ})$. In particular, $F_{\lambda}$ is globally exact, provided this function is bounded
below for some $\lambda \ge 0$, and the set
$$
  C_{\alpha} = 
  \{ x \in \mathcal{H} \mid f_0(x) < f_* + \alpha, \: f_i(x) < \alpha, \: i \in I, \: |f_j(x)| < \alpha, \: j \in J \}
$$
is bounded for some $\alpha > 0$.
\end{theorem}

\begin{proof}
Let $x_*$ be a globally optimal solution of  $(\mathcal{P}_{IJ})$. Note that the function $f_0$ is locally
Lipschitz continuous, since it is a finite DC function. Consequently, taking into account \eqref{PenTermErrorBound}
and applying \cite[Thrm~2.4 and Prp.~2.7]{Dolgopolik_ExactPenFunc} one obtains that the function $F_{\lambda}$
is locally exact at $x_*$, i.e. there exist $\lambda_*(x_*) \ge 0$ and a neighbourhood $U$ of $x_*$ such that
$F_{\lambda}(x) \ge F_{\lambda}(x_*)$ for all $x \in U$ and $\lambda \ge \lambda_*(x_*)$. Then applying the
localization principle for linear penalty functions (see~\cite[Thrm.~3.17]{Dolgopolik_ExactPenFunc}
and \cite[Thrm~4.1]{Dolgopolik_UnifiedExact}) one gets that the function $F_{\lambda}$ is globally exact if and
only if there exists $\lambda \ge 0$ such that the set $\{ x \in \mathcal{H} \mid F_{\lambda}(x) < f_* \}$ is either
bounded or empty. 

Suppose that $F_{\lambda_0}$ is bounded below for some $\lambda_0 \ge 0$, and the set $C_{\alpha}$ is bounded
for some $\alpha > 0$. Let us check that in this case $\{ x \mid F_{\lambda}(x) < f_* \} \subset C_{\alpha}$ for any
sufficiently large $\lambda$.

Indeed, if $x \notin C_{\alpha}$, then either $f_0(x) \ge f_* + \alpha$ or $\varphi(x) \ge \alpha$. In the former case
one has $F_{\lambda}(x) \ge f_0(x) > f_*$ for any $\lambda \ge 0$, while in the latter case one has
$$
  F_{\lambda}(x) = 
  F_{\lambda_0}(x) + (\lambda - \lambda_0) \varphi(x) \ge c + (\lambda - \lambda_0) \alpha > f_*
$$
for all $\lambda > \lambda_0 + (f_* - c) / \alpha$, where $c = \inf_{x \in \mathcal{H}} F_{\lambda_0}(x)$. Thus, 
$\{ x \mid F_{\lambda}(x) < f_* \} \subset C_{\alpha}$ for any $\lambda > \lambda_0 + (f_* - c) / \alpha$.	
\end{proof}

\begin{remark}
{(i)~Our proof of the global exactness of the $\ell_1$ penalty function is based on the assumption that the penalty
term $\varphi(x)$ has a local error bound. This assumption can be verified with the use of general results on metric
subregularity and local error bounds \cite{Aze,Gfrerer,Kruger2014}. In particular, in the case when the functions $f_i$
are continuously differentiable at a globally optimal solution $x_*$ of $(\mathcal{P}_{IJ})$, the function
$\varphi(x)$ has a local error bound at this optimal solution, provided MFCQ holds at $x_*$ (see, e.g.
\cite[Corollary~2.2]{Cominetti}). Let us note that in some cases it is possible to prove the existence of a local error
bound with the use of the DC structure of the problem alone (i.e. without any constraint qualifications). See
\cite{ThisDingNgai} for this kind of results on exact penalty functions and error bounds for DC optimization
problems with inequality constraints.
}

\noindent{(ii)~Note that Theorem~\ref{Thrm_GlobalExactPenFunc_DCProblem} significantly improves
\cite[Proposition~1]{Strekalovsky2017_2}, since we do not assume that the objective function $f_0$ is globally
Lipschitz continuous, and utilise a local error bound instead of the global one in \cite{Strekalovsky2017_2}.
Furthermore, we obtained necessary and sufficient conditions for the global exactness of the function
$F_{\lambda}(x)$, while only sufficient conditions were considered in \cite{Strekalovsky2017_2}.
}
\end{remark}

Applying the global optimality conditions from Theorem~\ref{Thrm_GlobOptCond} to $F_{\lambda}(x)$ one can easily obtain
new necessary and sufficient global optimality conditions for the problem $(\mathcal{P}_{IJ})$ that are valid under the
assumptions of Theorem~\ref{Thrm_GlobalExactPenFunc_DCProblem}. Namely, the following result holds true.

\begin{theorem} \label{Theorem_GlobOptCond_BothTypesConstr}
Let $\mathcal{H}$ be finite dimensional. Suppose that $\varphi$ has a local error bound at every globally optimal
solution of the problem $(\mathcal{P}_{IJ})$, and there exists $\lambda \ge 0$ such that the set
$\{ x \in \mathcal{H} \mid F_{\lambda}(x) < f_* \}$ is either bounded or empty. Suppose also that $x_*$ is a feasible
point of the problem $(\mathcal{P}_{IJ})$, $D f_k$ is a global codifferential of the function $f_k$, 
$C_k \subseteq \overline{d} f_k(x_*)$ is a nonempty set such that $\overline{d} f_k(x_*) = \cl \co C_k$, 
$k \in I \cup J \cup \{ 0 \}$, and $D_j \subseteq \underline{d} f_j(x_*)$ is a nonempty set such that 
$\underline{d} f_j(x_*) = \cl \co D_j$, $j \in J$. Then $x_*$ is a globally optimal solution of the problem
$(\mathcal{P}_{IJ})$ if and only if there exists $\lambda \ge 0$ such that for any $z_k \in C_k$, 
$k \in I \cup J \cup \{ 0 \}$ and $u_j \in D_j$, $j \in J$, one has
$a(z) \ge 0$, where $(a(z), v(z))$ with $z = (z_0, z_1, \ldots, z_m, u_{l + 1}, \ldots, u_m)$ is a globally optimal
solution of the problem
$$
  \min_{(a, v) \in \mathbb{R} \times \mathcal{H}} \| (a, v) \|^2 \quad
  \text{subject to} \quad (a, v) \in Q_{\lambda}(z),
$$
where
\begin{align*}
  Q_{\lambda}(z) = \cl\bigg\{ &\underline{d} f_0(x_*) + \lambda \sum_{i \in I} 
  \co\big\{ (f_i(x_*), 0) + \underline{d} f_i(x_*), - \overline{d} f_i(x_*) \mid i \in I \big\} \\
  &+ \lambda \sum_{j \in J} \co\big\{ \underline{d} f_j(x_*) + \underline{d} f_j(x_*), 
  - \overline{d} f_j(x_*) - \overline{d} f_j(x_*) \mid j \in J\big\} \bigg\} \\
  &+ z_0 + \lambda \sum_{i \in I} z_i + \lambda \sum_{j \in J} (z_j - u_j).
\end{align*}
\end{theorem}

\begin{proof}
As was noted above, $F_{\lambda}$ is a DC function. With the use of Proposition~\ref{Prp_GlobalCodiffCalc} one
can verify that the pair $D F_{\lambda} = [\underline{d} F_{\lambda}, \overline{d} F_{\lambda}]$ with
\begin{gather*}
  \underline{d} F_{\lambda}(x) = \cl\bigg\{ \underline{d} f_0(x) 
  + \lambda \sum_{i \in I} \co\big\{ (f_i(x), 0) + \underline{d} f_i(x), - \overline{d} f_i(x) \mid i \in I \big\} \\
  + \lambda \sum_{j \in J} \co\big\{ \underline{d} f_j(x) + \underline{d} f_j(x), 
  - \overline{d} f_j(x) - \overline{d} f_j(x) \mid j \in J\big\} \bigg\}, \\
  \overline{d} F_{\lambda}(x) = \cl\bigg\{ \overline{d} f_0(x) + \lambda \sum_{i \in I} \overline{d} f_i(x) 
  + \lambda \sum_{j \in J} \big( \overline{d} f_j(x) - \underline{d} f_j(x) \big) \bigg\}
\end{gather*}
for any feasible point $x$ is a global codifferential of $F_{\lambda}$. Observe also that 
$$
  \overline{d} F_{\lambda}(x_*) = \cl\co\Big\{ C_0 + \lambda \sum_{i \in I} C_i 
  + \lambda \sum_{j \in J} \big( C_j - D_j \big) \Big\}
$$
by the definitions of the sets $C_k$ and $D_j$.

By Theorem~\ref{Thrm_GlobalExactPenFunc_DCProblem} the function $F_{\lambda}$ is globally exact. Therefore, if $x_*$ is
a globally optimal solution of the problem $(\mathcal{P}_{IJ})$, then it is a point of global minimum of
$F_{\lambda}$ for any sufficiently large $\lambda$. Now, applying the global optimality conditions from
Theorem~\ref{Thrm_GlobOptCond} to the function $F_{\lambda}$ at $x_*$ with $\lambda$ large enough one obtains that the
``only if'' part of the theorem is valid.

Coversely, if there exists $\lambda \ge 0$ such that $a(z) \ge 0$ for any $z$ from the formulation of the theorem, then
by Theorem~\ref{Thrm_GlobOptCond} the point $x_*$ is a global minimizer of $F_{\lambda}$. Hence taking into account the
facts that $x_*$ is a feasible point of $(\mathcal{P}_{IJ})$ and $F_{\lambda}(x) = f_0(x)$ for any feasible point $x$
of this problem one gets that $x_*$ is a globally optimal solution of the problem $(\mathcal{P}_{IJ})$. Thus, the
``if'' part of the theorem is valid as well.	
\end{proof}

Let us consider two simple examples illustrating Theorems~\ref{Thrm_GlobalExactPenFunc_DCProblem} and
\ref{Theorem_GlobOptCond_BothTypesConstr}. The first example allows one to compare exact penalty approach
with ``interior point'' approach from Theorem~\ref{Theorem_GlobOptCond_InequalConst}, while in the second example we
analyse an equality constrained problem.

\begin{example}
Let us consider the same problem as in Example~\ref{Example_InequalConstrProblem}, i.e. the problem
\begin{equation} \label{ExProblem_InequalConstr_2}
  \min_{x \in \mathbb{R}} f_0(x) = |x - 4| \quad \text{subject to} \quad
  f_1(x) = \min\{ |x - 2|, |x + 2| \} - 1 \le 0.
\end{equation}
In this case the merit function $F_{\lambda}$ has the form
$$
  F_{\lambda}(x) = |x - 4| + \lambda \max\big\{ 0, \min\{ |x - 2|, |x + 2| \} - 1 \big\}.
$$
It is easily seen that this function is globally exact, and its least exact penalty parameter is equal to $1$. We set
$\lambda = 2$. Furthermore, one can check that for any $\lambda \ge 1$ the point $x_0 = -1$ is a local minimizer of
$F_{\lambda}$, i.e. $F_{\lambda}$ is locally exact at $x_0$. Let us apply the global optimality conditions from
Theorem~\ref{Thrm_GlobOptCond} to the function $F_{\lambda}$ at the point $x_0$, as it is done in
Theorem~\ref{Theorem_GlobOptCond_BothTypesConstr}.

Let, as above, $\varphi(x) = \max\{ 0, f_1(x_0) \}$. Applying Proposition~\ref{Prp_GlobalCodiffCalc} one gets 
\begin{align*}
  D \varphi(x_0) &= \Big[ \co\big\{ - \overline{d} f_1(x_0), \underline{d} f_1(x_0) \big\}, 
  \overline{d} f_1(x_0) \Big], \\
  D F_2(x_0) &= \Big[ \underline{d} f_0(x_0) + 2 \underline{d} \varphi(x_0),
  \overline{d} f_0(x_0) + 2 \overline{d} \varphi(x_0) \Big]
\end{align*}
Recall that global codifferentials of the functions $f_0$ and $f_1$ at $x_0$ were computed in
Example~\ref{Example_InequalConstrProblem}. Therefore, with the use of Example~\ref{Example_InequalConstrProblem} one
gets that
\begin{multline*}
  \underline{d} F_2(x_0) = \co\Bigg\{ \begin{pmatrix} -22 \\ 5 \end{pmatrix}, \begin{pmatrix} -26 \\ 1 \end{pmatrix},
  \begin{pmatrix} -10 \\ 1 \end{pmatrix}, \begin{pmatrix} -14 \\ -3 \end{pmatrix}, 
  \begin{pmatrix} -14 \\ 3 \end{pmatrix}, \begin{pmatrix} - 18 \\ -1 \end{pmatrix}, 
  \begin{pmatrix} -22 \\ 3 \end{pmatrix}, \\
  \begin{pmatrix} -10 \\ -1 \end{pmatrix},
  \begin{pmatrix} -12 \\ 3 \end{pmatrix}, \begin{pmatrix} -16 \\ -1 \end{pmatrix}, 
  \begin{pmatrix} 0 \\ -1 \end{pmatrix}, \begin{pmatrix} -4 \\ -5 \end{pmatrix}, 
  \begin{pmatrix} -4 \\ 1 \end{pmatrix}, \begin{pmatrix} -8 \\ -3 \end{pmatrix},
  \begin{pmatrix} -12 \\ 1 \end{pmatrix}, \begin{pmatrix} 0 \\ -3 \end{pmatrix}
  \Bigg\},
\end{multline*}
and
$$
  \overline{d} F_2(x_0) = \co\left\{ \begin{pmatrix} 4 \\ -2 \end{pmatrix}, \begin{pmatrix} 8 \\ 2 \end{pmatrix},
  \begin{pmatrix} 12 \\ -2 \end{pmatrix}, \begin{pmatrix} 0 \\ 2 \end{pmatrix} \right\}.
$$
Let $C$ be the set of extreme points of $\overline{d} F_2(0)$. Then one can check that
\begin{enumerate}
\item{$0 \in \underline{d} F_0(x_0) + z$ for $z = (8, 2) \in C$, $z = (12, -2) \in C$, and $z = (0, 2) \in C$;
}

\item{$(a(z), v(z)) = (- 4 / 17, - 16 / 17)$ for $z = (4, -2) \in C$.
}
\end{enumerate}
Thus, by Theorem~\ref{Thrm_GlobOptCond} the point $x_0$ is not a point of global minimum of the function
$F_2(x)$ and, therefore, is not a globally optimal solution of problem \eqref{ExProblem_InequalConstr_2}. However, note
that for $z = (4, -2)$ one has $x_1 = x_0 + a(z)^{-1} v(z) = -3$, i.e. $x_1$ is a globally
optimal solution of problem \eqref{ExProblem_InequalConstr_2} (cf. Example~\ref{Example_InequalConstrProblem}).
\end{example}

\begin{example}
Let $\mathcal{H} = \mathbb{R}^2$. Consider the following optimization problem:
\begin{equation} \label{ExProblem_EqConst}
  \min_{x \in \mathbb{R}^2} f_0(x) = |x_1 - 2| + 2 |x_2| \quad 
  \text{subject to} \quad f_1(x) = |x_1| - |x_2| = 0.
\end{equation}
The merit function $F_{\lambda}$ for this problem has the form
$$
  F_{\lambda}(x) = |x_1 - 2| + 2 |x_2| + \lambda \big| |x_1| - |x_2| \big|.
$$
It is easily seen that that the penalty term $\varphi(x) = | |x_1| - |x_2| |$ has a local error bound at the unique
globally optimal solution $x_* = (0, 0)$ of problem \eqref{ExProblem_EqConst}. Consequently, taking into account 
the fact that $f_0(x) \to + \infty$ as $\| x \| \to + \infty$ one obtains that the function $F_{\lambda}$ is globally
exact for problem \eqref{ExProblem_EqConst}. Let us estimate the least exact penalty parameter of $F_{\lambda}$.

One can easily verify that the function $f_0$ is globally Lipschitz continuous with Lipschitz constant $L = \sqrt{5}$,
and 
$$
  \varphi^{\downarrow}(x) = \liminf_{y \to x} \frac{\varphi(y) - \varphi(x)}{\| y - x \|} \le - 1
  \quad \forall x \notin \Omega,
$$
where $\| \cdot \|$ is the Euclidean norm. The quantity $\varphi^{\downarrow}(x)$ is called \textit{the rate of
steepest descent} of $\varphi$ at $x$ (see, e.g.~\cite{Demyanov_2000,Dolgopolik_ExactPenFunc_II}). For any 
$\lambda > \sqrt{5}$ and $x \notin \Omega$ one has 
$F_{\lambda}^{\downarrow}(x) \le L + \lambda \varphi^{\downarrow}(x) < 0$. Therefore, local/global minimizers of the
function $F_{\lambda}$ do not belong to the set $\mathbb{R}^2 \setminus \Omega$ for any $\lambda > \sqrt{5}$, since
$F^{\downarrow}_{\lambda}(x) \ge 0$ is a necessary optimality condition. Thus, one can conclude that the least
exact penalty parameter of $F_{\lambda}$ does not exceed $\sqrt{5}$. That is why we set $\lambda = 3$.

Let us apply the global optimality conditions from Theorem~\ref{Thrm_GlobOptCond} to the function $F_{\lambda}$
at the point $x_0 = (2, 0)$, which is infeasible for problem \eqref{ExProblem_EqConst} and is a point of unconstrained
global minimum of the objective function $f_0$. With the use of Proposition~\ref{Prp_GlobalCodiffCalc} one obtains that
\begin{gather*}
  \underline{d} f_0(x_0) = \co\left\{ \begin{pmatrix} 0 \\ 1 \\ 2 \end{pmatrix}, 
  \begin{pmatrix} 0 \\ 1 \\ -2 \end{pmatrix}, \begin{pmatrix} 0 \\ -1 \\ 2 \end{pmatrix},
  \begin{pmatrix} 0 \\ -1 \\ -2 \end{pmatrix} \right\}, \quad \overline{d} f_0(x_0) = \{ 0 \}, \\
  \underline{d} f_1(x_0) = \co\left\{ \begin{pmatrix} 0 \\ 1 \\ 0 \end{pmatrix},
  \begin{pmatrix} -4 \\ -1 \\ 0 \end{pmatrix} \right\}, \quad
  \overline{d} f_1(x_0) = \co\left\{ \begin{pmatrix} 0 \\ 0 \\ 1 \end{pmatrix},
  \begin{pmatrix} 0 \\ 0 \\ -1 \end{pmatrix}
  \right\}.
\end{gather*}
Furthermore, one has $D F_3(x_0) = [ \underline{d} f_0(x_0) + 3 \underline{d} \varphi(x_0), 
  \overline{d} f_0(x_0) + 3 \overline{d} \varphi(x_0) ]$, where
$$
  \underline{d} \varphi(x_0) = \co\left\{\underline{d} f_1(x_0) + \underline{d} f_1(x_0), 
   \begin{pmatrix} -4 \\ 0 \\ 0 \end{pmatrix} - \overline{d} f_1(x_0) - \overline{d} f_1(x_0) \right\},
$$
and $\overline{d} \varphi(x_0) = \overline{d} f_1(x_0) - \underline{d} f_1(x_0)$. Utilising these expressions for
global codifferentials one can easily compute $\underline{d} F_3(x_0)$, which is the convex hull of $20$ points and we
do not present it here for the sake of shortness, and check that
$$
  \overline{d} F_3(x_0) = \co\left\{ \begin{pmatrix} 0 \\ -3 \\ 3 \end{pmatrix},
  \begin{pmatrix} 12 \\ 3 \\ 3 \end{pmatrix}, \begin{pmatrix} 0 \\ -3 \\ -3 \end{pmatrix},
  \begin{pmatrix} 12 \\ 3 \\ -3 \end{pmatrix} \right\}.
$$
Let $C$ be the set of extreme points of $\overline{d} F_3(x_0)$. Then solving the problem
$$
  \min_{(a, v) \in \mathbb{R} \times \mathbb{R}^2} \| (a, v) \|^2 \quad 
  \text{subject to} \quad (a, v) \in \underline{d} F_3(x_0) + z
$$
one can check that
\begin{enumerate}
\item{$(a(z), v(z)) = (-1, 1.5, 0.5)$ for $z = (0, -3, 3) \in C$;
}

\item{$(a(z), v(z)) = (-0.8, 1.6, 0)$ for $z = (12, 3, 3) \in C$ and $z = (12, 3, -3) \in C$;
}

\item{$(a(z), v(z)) = (-1, 1.5, -0.5)$ for $z = (0, -3, -3) \in C$.
}
\end{enumerate}
Thus, the global optimality conditions from Theorem~\ref{Thrm_GlobOptCond} are not satisfeid at $x_0$. Moreover,
observe that for $z = (12, 3, 3) \in C$ and $z = (12, 3, -3) \in C$ one has 
$x_1 = x_0 + a(z)^{-1} v(z) = (0, 0)$, and $x_1$ is a globally optimal solution of \eqref{ExProblem_EqConst}.
\end{example}

As was noted above, in many particular cases the global optimality conditions from Theorems~\ref{Thrm_GlobOptCond},
\ref{Theorem_GlobOptCond_InequalConst}, and \ref{Theorem_GlobOptCond_BothTypesConstr} are of theoretical value only,
since it is extremely difficult (if at all possible) to compute a global codifferential of a general DC function and
verify the global optimality conditions. Apparently, our optimality conditions can be readily checked only in 
the piecewise affine case, i.e. when the global codifferential is a pair of convex polytopes (see
Example~\ref{Example_SupposrtSet_PolyhedralFunc}). Nevertheless, it seems possible to design new numerical methods for
DC optimization problems based on the global optimality conditions obtained in this article and utilising certain
polyhedral approximations of global codifferentials (cf. codifferential method in \cite{BagUgon}, aggregate
codifferential method in \cite{TorBagKar}, and a method based on successive piecewise-affine approximations in
\cite{GaudiosoBagirov}). A design and analysis of such numerical methods lie outside the scope of
this article. Here we only present a simple example of the usage of approximations of global codifferentials without
trying to outline the idea behind possible numerical methods.

Let $f$ be a DC function, and $D f$ be its global codifferential associated with a DC decomposition $f = g - h$. From
Proposition~\ref{Theorem_AffSupport_vs_Conjugate} and the definition of global codifferential
\eqref{GlobalHypodiffDefViaSupportSet} it follows that for any $x, y \in \mathcal{H}$, 
$v \in \partial g(y)$, and $w \in \partial h(y)$ one has 
$(g(y) - g(x) + \langle v, x - y \rangle, v) \in \underline{d} f(x)$ and 
$(h(x) - h(y) + \langle w, y - x \rangle, -w) \in \overline{d} f(x)$. Therefore, if a point $x$ is fixed, one can
choose sampling points $x_1, \ldots, x_m \in \mathcal{H}$, compute $v_k \in \partial g(x_k)$ and
$w_k \in \partial h(x_k)$, $k \in \{ 1, \ldots, m \}$, and consider the following inner approximations of the global
hypodifferential and the global hyperdifferential at $x$ respectively:
\begin{equation} \label{InnerApprox_GlobalCodiff}
\begin{split}
  \co\big\{ (g(x_k) - g(x) + \langle v_k, x - x_k \rangle, v_k) \bigm| k \in \{ 1, \ldots, m \} \big\} 
  &\subset \underline{d} f(x), \\
  \co\big\{ (h(x) - h(x_k) + \langle w, x_k - x \rangle, -w_k) \bigm| k \in \{ 1, \ldots, m \} \big\} 
  &\subset \overline{d} f(x).
\end{split}
\end{equation}
The following example demonstrates that even if these inner approximations are very crude, one can still utilise
them along with the global optimality conditions from Theorems~\ref{Thrm_GlobOptCond},
\ref{Theorem_GlobOptCond_InequalConst}, and \ref{Theorem_GlobOptCond_BothTypesConstr} to escape from a local minimum
(or stationary point). 

\begin{example} \label{Example_NonPiecewiseAff_1}
Let $\mathcal{H} = \mathbb{R}^2$. Consider the following DC optimization problem:
\begin{equation} \label{NonconvexQuadraticProblem}
  \min_{x \in \mathbb{R}^2} f_0(x) = x_1^2 - x_2^2 \quad 
  \text{subject to} \quad -1 \le x_2 \le 2.
\end{equation}
We rewrite the constraints of this problem as follows: $f_1(x) = x_2 - 2 \le 0$ and $f_2(x) = -x_2 - 1 \le 0$. Let 
$x_0 = (0, -1)$. Clearly, $x_0$ is a locally (but not globally) optimal solution of problem
\eqref{NonconvexQuadraticProblem}. We would like to escape from this local minimum. 

The merit function $F_{\lambda}(x) = f_0(x) + \lambda( \max\{0, x_2 - 2\} + \max\{ 0, - x_2 - 1 \})$ for 
problem \eqref{NonconvexQuadraticProblem} is not bounded below for any $\lambda \ge 0$. Therefore, we will use the
``interior point'' approach of Theorem~\ref{Theorem_GlobOptCond_InequalConst}. With the use of
Example~\ref{Example_SupportSetComputation} one obtains that the global codifferential 
$D f = [\underline{d} f, \overline{d} f]$ of the function $f_0$ associated with the DC decomposition $f_0 = g_0 - h_0$,
where $g_0(x) = x_1^2$ and $h_0(x) = x_2^2$, has the form:
$$
  \underline{d} f_0(x) = \co\left\{ \begin{pmatrix} - (x_1 - y_1)^2 \\ 2 y_1 \\ 0 \end{pmatrix},
  y_1 \in \mathbb{R} \right\}, 
  \overline{d} f_0(x) = \co\left\{ \begin{pmatrix} (x_2 - y_2)^2 \\ 0 \\ - 2 y_2 \end{pmatrix},
  y_2 \in \mathbb{R} \right\}.
$$
A direct usage of this global codifferential leads to rather cumbersome and complicated computations. That is why we
will use inner approximations \eqref{InnerApprox_GlobalCodiff} instead. As sampling points we choose five points: $x_0$
and $x_0 + \xi_{ij}$, $i, j \in \{ 1, 2 \}$, where $\xi_{ij} = ( (-1)^i 2, (-1)^j 2)$ are the extereme points of the
ball of radius $2$ in the $\ell_{\infty}$ norm. Applying \eqref{InnerApprox_GlobalCodiff} one gets the following inner
approximations:
$$
  \co\left\{ \begin{pmatrix} 0 \\ 0 \\ 0 \end{pmatrix},
  \begin{pmatrix} -4 \\ -4 \\ 0 \end{pmatrix}, \begin{pmatrix} -4 \\ 4 \\0 \end{pmatrix} \right\} \subset 
  \underline{d} f_0(x_0), \:
  \co\left\{ \begin{pmatrix} 0 \\ 0 \\ 2 \end{pmatrix},
  \begin{pmatrix} 4 \\ 0 \\ -2 \end{pmatrix}, \begin{pmatrix} 4 \\ 0 \\ 6 \end{pmatrix} \right\} \subset 
  \overline{d} f_0(x_0).
$$
Note also that $D f_1(x_0) = [ \{(0, 0, 1) \}, \{ 0 \} ]$ and $D f_2(x_0) = [ \{(0, 0, -1) \}, \{ 0 \} ]$.

Let us apply the global optimality conditions from Theorem~\ref{Theorem_GlobOptCond_InequalConst}. In our case these
conditions take the form: for any $z \in \overline{d} f_0(x_0)$ one has $a(z) \ge 0$, where $(a(z), v(z))$ is a globally
optimal solution of the problem
$$
  \min_{(a, v) \in \mathbb{R} \times \mathbb{R}^2} \| (a, v) \|^2 \quad 
  \text{s.t.} \quad (a, v) \in 
  L(z) = \cl\co\big\{ \underline{d} f_0(x_0) + z, \underline{d} f_1(x_0), \underline{d} f_2(x_0) \big\}.
$$
Replacing $\underline{d} f_0(x_0)$ with its inner approximation computed above one obtains that
$(a(z), v(z)) \approx (-0.1034, 0, -0.2414)$ for $z = (4, 0, -2) \in \overline{d} f_0(x_0)$, i.e. the optimality
conditions are not satisfied for the approximation. Following Remark~\ref{Remark_InteriorPointStep} define
$\overline{x} = x_0 + (a(z))^{-1} v(z) \approx (0, 1.333)$. Note that $\overline{x}$ belongs to the interior of the
feasible region, and $f(\overline{x}) \approx -1.777 < -1 = f(x_0)$. Thus, the use of inner approximations
\eqref{InnerApprox_GlobalCodiff} helped us escape from the local minimum.
\end{example}

\section{Some connections between optimality conditions}
\label{Sect_OtherOptimalityCond}

Let us point out some connections between global optimality conditions obtained in the previous section, well-known
global optimality conditions in terms of $\varepsilon$-subdifferentials, and KKT optimality conditions. 

We start with global optimality conditions in terms of $\varepsilon$-subdifferentials. Let us consider the unconstrained
DC optimization problem
$$
  \min_{x \in \mathcal{H}} f(x) = g(x) - h(x),	\eqno{(\mathcal{P}_0)}
$$
where $g$ and $h$ are finite closed convex functions. Suppose that $f$ is bounded below. Recall that $x_*$ is a point
of global minimum of the function $f$ if and only if 
$$
  \partial_{\varepsilon} h(x_*) \subseteq \partial_{\varepsilon} g(x_*) \quad \forall \varepsilon \ge 0
$$
(see \cite{HiriartUrruty_GlobalDC}). On the other hand, by Theorem~\ref{Thrm_GlobOptCond} the point $x_*$ is a globally
optimal solution of the problem $(\mathcal{P}_0)$ if and only if for any $z \in \overline{d} f(x_*)$ one has $a(z) \ge
0$, where $(a(z), v(z))$ is a globally optimal solution of the problem
$$
  \min_{(a, v) \in \mathbb{R} \times \mathcal{H}} \| (a, v) \|^2 \quad
  \text{subject to} \quad (a, v) \in \underline{d} f(x_*) + z,
$$
and $D f$ is a global codifferential of $f$ associated with the DC decomposition $f = g - h$. Let us point out a
direct connection between these optimality conditions.

\begin{theorem}
Let $x_* \in \mathcal{H}$ and $\varepsilon_0 \ge 0$ be given. Then 
$\partial_{\varepsilon} h(x_*) \subseteq \partial_{\varepsilon} g(x_*)$ for any $\varepsilon \le \varepsilon_0$ if and
only if for any $z = (b, w) \in \overline{d} f(x_*)$ with $b \le \varepsilon_0$ one has $a(z) \ge 0$.
\end{theorem}

\begin{proof}
Fix any $z = (b, w) \in \overline{d} f(x_*)$ such that $b \le \varepsilon$ for some $\varepsilon \ge 0$. By the last
part of Proposition~\ref{Thrm_MinViaGlobalHypodiff} and the definition of global codifferential
\eqref{GlobalHypodiffDefViaSupportSet} (see also~\eqref{DC_ConvexPartViaGlobalCodiff}) one obtains that $a(z) \ge 0$ if
and only if
\begin{equation} \label{b-Subdiff_Def}
  g(x) - g(x_*) + b + \langle w, x - x_* \rangle \ge 0 \quad \forall x \in \mathcal{H}.
\end{equation}
Hence with the use of the inequality $b \le \varepsilon$ one gets that $- w \in \partial_{\varepsilon} g(x_*)$. Observe
also that 
$$
  \partial_{\varepsilon} h(x_*) = \{ w \in \mathcal{H} \mid 
  \exists b \in [0, \varepsilon] \colon (b, - w) \in \overline{d} f(x_*) \}
$$
by Proposition~\ref{Thrm_SubdiffViaSupportSet} and the definition of global codifferential. 

Suppose that $a(z) \ge 0$ for any $z = (b, w) \in \overline{d} f(x_*)$ with $b \le \varepsilon_0$, and fix any 
$\varepsilon \le \varepsilon_0$ and $w \in \partial_{\varepsilon} h(x_*)$. Then $z = (b, - w) \in \overline{d} f(x_*)$
for some $b \le \varepsilon$, which implies that $a(z) \ge 0$ and $w \in \partial_{\varepsilon} g(x_*)$. Thus,
$\partial_{\varepsilon} h(x_*) \subseteq \partial_{\varepsilon} g(x_*)$ for any $\varepsilon \le \varepsilon_0$,

Conversely, suppose that $\partial_{\varepsilon} h(x_*) \subseteq \partial_{\varepsilon} g(x_*)$ for any 
$\varepsilon \le \varepsilon_0$. Choose any $z = (b, w) \in \overline{d} f(x_*)$ with $b \le \varepsilon_0$. Then
$- w \in \partial_b h(x_*)$, which implies that \eqref{b-Subdiff_Def} holds true. Hence $a(z) \ge 0$ by
Proposition~\ref{Thrm_MinViaGlobalHypodiff}, and the proof is complete.	
\end{proof}

Let us also consider the reverse convex minimization problem of the form:
$$
  \min_{x \in \mathcal{H}} f_0(x) \quad \text{subject to} \quad g(x) \ge 0. \eqno{(\mathcal{P}_1)}
$$
Here $f_0$ and $g$ are finite closed convex functions. We suppose that the feasible region of this problem is
nonempty, and there exists an infeasible point $x_0$ such that 
\begin{equation} \label{ReverseConvexProblem_SlaterCond}
  f_0(x_0) < \inf\{ f_0(x) \mid x \in \mathcal{H} \colon g(x) \ge 0 \},
\end{equation}
i.e. the optimal value of the problem $(\mathcal{P}_1)$ is strictly grater than the infimum of $f_0$ over the entire
space $\mathcal{H}$. Note that if $x_* \in \mathcal{H}$ is a globally optimal solution of $(\mathcal{P}_1)$ and 
$g(x) > 0$, then $x_*$ is obviously a point of unconstrained local minimum of the function $f_0$, which due to the
convexity of $f_0$ implies that $x_*$ is a global minimizer of $f_0$ and $0 \in \partial f_0(x_*)$. Therefore, below we
suppose that all global minimizers $x_*$ of $(\mathcal{P}_1)$ satisfy the equality $g(x_*) = 0$.

Recall that a feasible point $x_*$ is a globally optimal solution of $(\mathcal{P}_1)$ if and only if
$\partial_{\varepsilon} g(x_*) \subset \bigcup_{\alpha \ge 0} \partial_{\varepsilon} (\alpha f_0)(x_*)$ 
for all $\varepsilon \ge 0$ by \cite[Theorem~3.5 and Remark~3.7]{HiriartUrruty_98}. On the other hand, rewriting the
constraint $g(x) \ge 0$ as $f_1(x) \le 0$ with $f_1(x) = - g(x)$ one can apply global optimality conditions from
Theorem~\ref{Theorem_GlobOptCond_InequalConst} to this problem. Let $D f_0$ be a global codifferential of $f_0$
associated with the DC decomposition $f_0 = f_0 - 0$, while $D f_1$ be a global codifferential of $f_1$ associated
with the DC decomposition $f_1 = 0 - g$. Clearly, $\overline{d} f_0(\cdot) \equiv \{ 0 \}$ and 
$\underline{d} f_1(\cdot) \equiv \{ 0 \}$. Consequently, under the assumptions of
Theorem~\ref{Theorem_GlobOptCond_InequalConst} a feasible point $x_*$ is a globally optimal solution of
$(\mathcal{P}_1)$ if and only if for any $z \in \overline{d} f_1(x_*)$ one has $a(z) \ge 0$, where
$(a(z), v(z))$ is a globally optimal solution of the problem
$$
  \min_{(a, v) \in \mathbb{R} \times \mathcal{H}} \| (a, v) \|^2 \quad
  \text{subject to} \quad (a, v) \in \cl\co\{\underline{d} f_0(x_*), z \}.
$$
Let us describe how these two optimality conditions are connected.

\begin{theorem} \label{Thrm_ReverseConvex_vs_Codifferential}
Let a feasible point $x_*$ of $(\mathcal{P}_1)$ and $\varepsilon_0 \ge 0$ be given. Then 
$\partial_{\varepsilon} g(x_*) \subset \bigcup_{\alpha \ge 0} \partial_{\varepsilon} (\alpha f_0)(x_*)$ for any 
$\varepsilon \le \varepsilon_0$ if and only if for any $z = (b, w) \in \overline{d} f_1(x_*)$ with 
$b \le \varepsilon_0$ one has $a(z) \ge 0$.
\end{theorem}

\begin{proof}
Fix any $z = (b, w) \in \overline{d} f_1(x_*)$. By the last part of Proposition~\ref{Thrm_MinViaGlobalHypodiff} and 
the definition of global codifferential (see \eqref{GlobalHypodiffDefViaSupportSet} and
\eqref{DC_ConvexPartViaGlobalCodiff}) one obtains that $a(z) \ge 0$ if and only if
\begin{equation} \label{EpsNormalSet_ReverseConvex}
  \max\big\{ f_0(x) - f_0(x_*), b + \langle w, x - x_* \rangle \big\} \ge 0 \quad \forall x \in \mathcal{H}.
\end{equation}
Define $C_0 = \{ x \in \mathcal{H} \mid f_0(x) < f_0(x_*) \}$ and 
$C = \{ x \in \mathcal{H} \mid f_0(x) \le f_0(x_*) \}$. From \eqref{ReverseConvexProblem_SlaterCond} it follows that
$C_0$ is nonempty, while by \cite[Proposition~VI.1.3.3]{Lemarechal} one has $C = \cl C_0$. Clearly, inequality
\eqref{EpsNormalSet_ReverseConvex} is satisfied iff $b + \langle w, x - x_* \rangle \ge 0$ for any
$x \in C_0$. In turn, this inequality is satisfied iff $\langle w, x - x_* \rangle \ge - b$ for all $x \in C$ due
to the fact that $C = \cl C_0$. Thus, $a(z) \ge 0$ for some $z = (b, w) \in \overline{d} f_1(x_*)$ iff 
$\langle -w, x - x_* \rangle \le b$ for all $x \in C$ or equivalently $- w \in N_b(C, x_*)$, where $N_b(C, x_*)$ is the
set of $b$-normal directions to the set $C$ at $x_*$ (see, e.g. \cite[Definition~XI.1.1.3]{LemarechalII}). Note also
that by Proposition~\ref{Thrm_SubdiffViaSupportSet} and the definition of global codifferential
\eqref{GlobalHypodiffDefViaSupportSet} for any $\varepsilon \ge 0$ one has 
$\partial_{\varepsilon} g(x_*) = \{ w \in \mathcal{H} \mid 
\exists b \in [0, \varepsilon] \colon (b, - w) \in \overline{d} f_1(x_*) \}$.

Suppose that for any $z = (b, w) \in \overline{d} f_1(x_*)$ with $b \le \varepsilon_0$ one has $a(z) \ge 0$, and fix
any $\varepsilon \le \varepsilon_0$ and $w \in \partial_{\varepsilon} g(x_*)$. Then 
$z = (b, - w) \in \overline{d} f(x_*)$ for some $b \le \varepsilon$, which implies that $a(z) \ge 0$ and 
$w \in N_b(C, x_*) \subseteq N_{\varepsilon}(C, x_*)$. Applying \cite[Corollary~XI.3.6.2]{LemarechalII} one gets that
$N_{\varepsilon}(C, x_*) = \bigcup_{\alpha \ge 0} \partial_{\varepsilon} (\alpha f_0)(x_*)$. Thus,
for any $\varepsilon \le \varepsilon_0$ one has 
$\partial_{\varepsilon} g(x_*) \subset \bigcup_{\alpha \ge 0} \partial_{\varepsilon} f_0(x_*)$.

Conversely, suppose that 
$\partial_{\varepsilon} g(x_*) \subset \bigcup_{\alpha \ge 0} \partial_{\varepsilon} (\alpha f_0)(x_*)$
for any $\varepsilon \le \varepsilon_0$. Choose any $z = (b, w) \in \overline{d} f_1(x_*)$ with $b \le \varepsilon_0$.
Then $- w \in \partial_b g(x_*) \subset \bigcup_{\alpha \ge 0} \partial_b (\alpha f_0)(x_*) = N_b(C, x_*)$, which, as we
proved above, is equivalent to the inequality $a(z) \ge 0$.	
\end{proof}

The theorem above can be further extended to the case of convex maximization problems of the form:
$$
  \max_{x \in \mathcal{H}} f(x) \quad \text{subject to} \quad x \in C.
  \eqno{(\mathcal{P}_{\max})}
$$
Here $C = \{ x \in \mathcal{H} \mid f_i(x) \le 0, \: i \in I = \{ 1, \ldots, l \} \}$, and $f$ and $f_i$, $i \in I$, are
finite closed convex functions. We suppose that Slater's condition holds true, and the infimum of $f$ over $C$ is
strictly smaller than the maximum. 

Recall that a feasible point $x_*$ is a globally optimal solution of $(\mathcal{P}_{\max})$ if and only if 
$\partial_{\varepsilon} f(x_*) \subset N_{\varepsilon}(C, x_*)$ for all $\varepsilon \ge 0$ by
\cite[Proposition~3.9]{HiriartUrruty_GlobalDC}. Recasting the problem
$(\mathcal{P}_{\max})$ as the problem of minimizing the function $f_0(x) = - f(x)$ over $C$ one can apply global
optimality conditions from Theorem~\ref{Theorem_GlobOptCond_InequalConst} to this problem. Let $D f_0$ be a global
codifferential of $f_0$ associated with the DC decomposition $f_0 = 0 - f$, while 
$D f_i$ be a global codifferential of $f_i$ associated with the DC decomposition $f_i = f_i - 0$, $i \in I$. Then
$\underline{d} f_0(\cdot) \equiv \{ 0 \}$ and 
$\overline{d} f_i(\cdot) \equiv \{ 0 \}$, $i \in I$. Therefore, by Theorem~\ref{Theorem_GlobOptCond_InequalConst} 
a feasible point $x_*$ is a globally optimal solution of $(\mathcal{P}_{\max})$ if and only if for any 
$z = (b, w) \in \overline{d} f_0(x_*)$ one has $a(z) \ge 0$, where $(a(z), v(z))$ is a globally optimal solution of the
problem
$$
  \min_{(a, v) \in \mathbb{R} \times \mathcal{H}} \| (a, v) \|^2 \quad
  \text{subject to} \quad (a, v) \in \cl\co\big\{ z, \underline{d} f_i(x_*) + (f_i(x_*), 0) \mid i \in I \big\}.
$$
From Proposition~\ref{Thrm_MinViaGlobalHypodiff} and the definition of global codifferential
\eqref{GlobalHypodiffDefViaSupportSet} it follows that $a(z) \ge 0$ iff
$$
  \max\big\{ b + \langle w, x - x_* \rangle, f_1(x), \ldots, f_l(x) \big\} \ge 0 \quad \forall x \in \mathcal{H}.
$$
In turn, this inequality is satisfied iff $- w \in N_b(C, x_*)$. Utilising this result and arguing in the same way as in
the proof of Theorem~\ref{Thrm_ReverseConvex_vs_Codifferential} one can easily check that the following connection
between the two global optimality conditions for the problem $(\mathcal{P}_{\max})$ exists.

\begin{theorem}
Let a feasible point $x_*$ of $(\mathcal{P}_{\max})$ and $\varepsilon_0 \ge 0$ be given. Then 
$\partial_{\varepsilon} f(x_*) \subset N_{\varepsilon}(C, x_*)$ for any $\varepsilon \le \varepsilon_0$ if and only if
for any $z = (b, w) \in \overline{d} f_0(x_*)$ with $b \le \varepsilon_0$ one has $a(z) \ge 0$.
\end{theorem}

Thus, one can say that there is an intimate relation between global optimality conditions for DC optimization problems
in terms of global codifferentials and in terms of $\varepsilon$-subdifferentials.

Now we turn to KKT optimality conditions. For the sake of simplicity, let us consider the inequality constrained problem
$$
  \min_{x \in \mathcal{H}} f_0(x) \quad \text{subject to} \quad f_i(x) \le 0, \quad i \in I = \{ 1, \ldots, l \},
  \eqno{(\mathcal{P}_I)}
$$
where $f_i = g_i - h_i$ are DC functions such that the convex functions $g_i$ and $h_i$ are differentiable, 
$i \in I \cup \{ 0 \}$. Let $D f_i$ be a global codifferential of $f_i$ associated with the DC decomposition 
$f_i = g_i - h_i$, $i \in I \cup \{ 0 \}$. Recall that under the assumptions of 
Theorem~\ref{Theorem_GlobOptCond_InequalConst} a feasible point $x_*$ is a globally optimal solution of the problem
$(\mathcal{P}_I)$ if and only if for any $z_i \in \overline{d} f_i(x_*)$, $i \in I \cup \{ 0 \}$, one has $a(z) \ge 0$,
where $(a(z), v(z))$ with $z = (z_0, z_1, \ldots, z_l)$ is a globally optimal solution of the problem
$$
  \min_{(a, v) \in \mathbb{R} \times \mathcal{H}} \| (a, v) \|^2 \quad 
  \text{subject to} \quad (a, v) \in L(z)
$$
with 
\begin{equation} \label{TrickPenaltyInequal_again}
  L(z) = \cl\co\{ \underline{d} f_0(x_*) + z_0, \: \underline{d} f_i(x_*) + z_i + (f_i(x_*), 0) \mid i \in I \}.
\end{equation}
(in the case $a(z) = 0$ one actually has $0 \in L(z)$; see Proposition~\ref{Thrm_MinViaGlobalHypodiff} and
Remark~\ref{Remark_MainThrm_GlobSuppSet}). Let us show that in the case when 
$z_i = (b_i, w_i) \in \overline{d} f_i(x_*)$ are such that $b_i = 0$, these optimality conditions are closely connected
to the KKT optimality conditions.

\begin{theorem}
Let $x_*$ be a feasible point of the problem $(\mathcal{P}_I)$, the function $f_0$ be bounded below on the feasible
region of this problem, and let MFCQ hold at $x_*$, i.e. there exists $y \in \mathcal{H}$ for which 
$\langle \nabla f_i(x_*), y \rangle < 0$ for any $i \in I$ such that $f_i(x_*) = 0$. Then KKT optimality conditions hold
true at $x_*$ if and only if for any $z_i = (b_i, w_i) \in \overline{d} f_i(x_*)$ with $b_i = 0$, 
$i \in I \cup \{ 0 \}$, one has $0 \in L(z)$, where $z = (z_0, z_1, \ldots, z_l)$.
\end{theorem}

\begin{proof}
Observe that $z_i = (0, w_i) \in \overline{d} f_i(x_*)$, $i \in I \cup \{ 0 \}$, if and only if 
$w_i = - \nabla h_i(x_*)$ by Proposition~\ref{Thrm_SubdiffViaSupportSet} and the definition of global codifferential
\eqref{GlobalHypodiffDefViaSupportSet}. Therefore, let $z = (z_0, z_1, \ldots, z_l)$ with
$z_i = (0, - \nabla h_i(x_*)) \in \overline{d} f_i(x_*)$ for all $i \in I \cup \{ 0 \}$.

From \eqref{TrickPenaltyInequal_again} and \eqref{DC_ConvexPartViaGlobalCodiff} it follows that the set $L(z)$ is an
affine support set of the function
\begin{align*}
  F(x) = \max_{i \in I} \big\{ &g_0(x) - g_0(x_*) - \langle \nabla h_0(x_*), x - x_* \rangle, \\
  &g_i(x) - g_i(x_*) + f_i(x_*) - \langle \nabla h_i(x_*), x - x_* \rangle \big\}.
\end{align*}
Clearly, $F(x_*) = 0$. Moreover, from the convexity of the functions $h_i$ it follows that
$h_i(x) - h_i(x_*) \ge \langle \nabla h_i(x_*), x - x_* \rangle$, which implies that for all $x \in \mathcal{H}$ the
inequality $F(x) \ge \max_{i \in I}\{ f_0(x) - f_0(x_*), f_i(x) \}$ holds true. Consequently, for any feasible point
$x$ one has $F(x) \ge f_0(x) - f_0(x^*)$, and $F(x) > 0$ otherwise, i.e. the function $F$ is bounded
below, since $f_0$ is bounded below on the feasible region by our assumption. Therefore, by the last part of
Proposition~\ref{Thrm_MinViaGlobalHypodiff} one gets that $0 \in L(z)$ if and only
if $F(x) \ge 0$ for all $x \in \mathcal{H}$. Hence taking into account the facts that $F$ is a convex function and
$F(x_*) = 0$ one obtains that $0 \in L(z)$ iff $0 \in \partial F(x_*)$. 

Bearing in mind the fact that $x_*$ is a feasible point of the problem $(\mathcal{P}_I)$, and
applying the well-known formula for the subdifferential of the maximum of a finite family of convex functions (see,
e.g. \cite[Corollary~VI.4.3.2]{Lemarechal}) one obtains that 
$$
  \partial F(x_*) = \co\big\{ \nabla f_0(x_*), \nabla f_i(x_*) \bigm| i \in I \colon f_i(x_*) = 0 \big\}.
$$
Thus, $0 \in L(z)$ if and only if there exist $\alpha_i \ge 0$, $i \in I \cup \{ 0 \}$, such that
$$
  \sum_{i = 0}^l \alpha_i \nabla f_i(x_*) = 0, \quad \sum_{i = 0}^l \alpha_i = 1,
$$
and $\alpha_i = 0$ whenever $f_i(x_*) < 0$. Note that $\alpha_0 \ne 0$ due to the fact that MFCQ holds at $x_*$, since
otherwise $0 = \sum_{i = 1}^l \alpha_i \langle \nabla f_i(x_*), y \rangle < 0$, which is impossible. Hence dividing by
$\alpha_0$ and denoting $\lambda_i = \alpha_i / \alpha_0$ one obtains that $0 \in L(z)$ iff there exist 
$\lambda_i \ge 0$, $i \in I$, such that
$$
  \nabla f_0(x_*) + \sum_{i = 1}^l \lambda_i \nabla f_i(x_*) = 0, \quad
  \lambda_i f_i(x_*) = 0, \quad \lambda_i \ge 0 \quad \forall i \in I,
$$
i.e. if and only if KKT optimality conditions are satisfied at $x_*$.	
\end{proof}

\section{A problem of Bolza}
\label{Sect_Bolza}

In some applications it might be extremely difficult to solve the problem 
\begin{equation} \label{GlobOptCond_WhenUseless}
  \min_{(a, v) \in \mathbb{R} \times \mathcal{H}} \: \| (a, v) \|^2 \quad 
  \text{subject to } (a, v) \in \underline{d} f(x_*) + z
\end{equation}
in order to find $(a(z), v(z))$, which renders the global optimality conditions presented above useless. The aim of this
section is to demonstrate that in this case one can utilise different global optimality condition in terms of global
codifferentials. Below we derive these conditions and apply them to a nonsmooth problem of Bolza.

\begin{theorem} \label{Thrm_GlobOptCond_Alternative}
Let $f$ be a DC function, $D f$ be any global codifferential of $f$, and $x_* \in \mathcal{H}$ be a given point. Suppose
that $C \subseteq \overline{d} f(x_*)$ is a nonempty set such that  $\overline{d} f(x_*) = \cl \co C$. Then $x_*$ is a
point of global minimum of the function $f$ if and only if for any $z \in C$ there exists $\xi(z) \ge 0$ such that
$(\xi(z), 0) \in \underline{d} f(x_*) + z$.
\end{theorem}

\begin{proof}
Clearly, $x_*$ is a point of global minimum of the function $f$ iff for any $z \in C$ the function
$$
  G_z(x) = \sup_{(a, v) \in \underline{d} f(x_*) + z} (a + \langle v, x \rangle)
$$
is nonnegative (note that $G_z(x) \ge f(x_* + x) - f(x_*)$). Applying the second part of
Proposition~\ref{Thrm_MinViaGlobalHypodiff} one obtains that if $x_*$ is a point of global minimum, then for any $z \in
C$
one has $(\xi(z), 0) \in \underline{d} f(x_*) + z$, where $\xi(z) = \inf_{x \in \mathcal{H}} G_z(x) \ge 0$. Conversely,
if for any $z \in C$ there exists $\xi(z) \ge 0$ such that $(\xi(z), 0) \in \underline{d} f(x_*) + z$, 
then $\inf_{x \in \mathcal{H}} G_z(x) \ge \xi(z) \ge 0$, and $x_*$ is a point of global minimum.	
\end{proof}

With the use of the first part of Proposition~\ref{Thrm_MinViaGlobalHypodiff} and the fact that by the definition of
global
codifferential $f(x) - f(x_*) = \inf_{z \in C} G_z(x - x_*)$ for all $x \in \mathcal{H}$ one can easily obtain 
the following result.

\begin{theorem} \label{Thrm_BoundednessBelow_GlobCodiff}
Let $f$ be a DC function, $D f$ be any global codifferential of $f$, and $x_* \in \mathcal{H}$ be a given point.
Suppose that $C \subseteq \overline{d} f(x_*)$ is a nonempty set such that $\overline{d} f(x_*) = \cl \co C$. Then $f$
is bounded below if and only if there exists $\xi \in \mathbb{R}$ such that
for any $z \in C$ one has $([\xi, + \infty) \times \{ 0 \}) \cap (\underline{d} f(x_*) + z) \ne \emptyset$.
\end{theorem}

Let us present an example in which global optimality conditions from Section~\ref{Sect_GlobOptCond} become too
complicated and unverifiable, while Theorems~\ref{Thrm_GlobOptCond_Alternative} and
\ref{Thrm_BoundednessBelow_GlobCodiff} can be easily applied. This example also demonstrates how one can compute a
global codifferential of a variational functional (cf.~\cite{Dolgopolik_CalcVar}). Namely, let us analyse the following
nonsmooth problem of Bolza:
\begin{equation} \label{BolzaProblem}
  \min \: \mathcal{I}(u) = u(0) - e^{-1} u(1) + \int_0^1 \max\big\{ |u'(x)| - |u(x)|, 0 \big\} \, dx.
\end{equation}
Here $u$ is from the Sobolev space $W^{1, 1}(0, 1)$. As was demonstrated in
\cite{IoffeRockafellar,Dolgopolik_CalcVar}, the function $u_*(x) = \theta e^x$ with $\theta > 0$ satisfies several
necessary optimality conditions for problem \eqref{BolzaProblem}. Our main goal is to demonstrate that this solution is
not globally optimal. Furthermore, we will show that the functional $\mathcal{I}(\cdot)$ is unbounded below and thus
does not attain a global minimum.

To convert the problem to the Hilbert space setting, below we suppose that $u \in H^1(0, 1) = W^{1, 2}(0, 1)$. Clearly,
if $u_*$ is not a globally optimal solution in $H^1(0, 1)$, then it is not a globally optimal solution in 
$W^{1, 1}(0, 1)$. Let us compute a global codifferential mapping of the restriction of the functional $\mathcal{I}$ to
the Hilbert space $H^1(0, 1)$. To this end, for any $x \in [0, 1]$ introduce the function
$$
  f_x(u, \xi) = \max\big\{ |\theta e^x + \xi| - |\theta e^x + u|, 0 \big\}
  = \max\big\{ |\theta e^x + \xi|, |\theta e^x + u| \big\} - |\theta e^x + u|.
$$
Applying Proposition~\ref{Prp_GlobalCodiffCalc} one obtains that the pair 
$D f_x(0, 0) = [\underline{d} f_x(0, 0), \overline{d} f_x(0, 0)]$ with
\begin{align} \label{IntegrandGlobalCodiff}
  \underline{d} f_x(0, 0) &= \co\left\{ \begin{pmatrix} 0 \\ 1 \\ 0 \end{pmatrix}, 
  \begin{pmatrix} - 2 \theta e^x \\ -1 \\ 0 \end{pmatrix}, \begin{pmatrix} 0 \\ 0 \\ 1 \end{pmatrix},
  \begin{pmatrix} - 2 \theta e^x \\ 0 \\ -1 \end{pmatrix} \right\}, \\
  \overline{d} f_x(0, 0) &= \co\left\{ \begin{pmatrix} 0 \\ -1 \\ 0 \end{pmatrix},
  \begin{pmatrix} 2 \theta e^x \\ 1 \\ 0 \end{pmatrix} \right\}. \label{IntegrandGlobalHyperdiff}
\end{align}
is a global codifferential of $f_x$ at $(0, 0)$. Then by the definition of global codifferential and the fact that
$f_x(0, 0) = 0$ one gets that
\begin{multline*}
  \mathcal{I}(u_* + u) - \mathcal{I}(u_*) = u(0) - e^{-1} u(1) 
  + \int_0^1 \Big( \max_{(a, v) \in \underline{d}_x f(0, 0)} \big(a + v_1 u(x) + v_2 u'(x) \big) \\
  + \min_{(b, w) \in \overline{d}_x f(0, 0)} \big( b + w_1 u(x) + w_2 u'(x) \big) \Big) \, dx
\end{multline*}
for any $u \in H^1(0, 1)$. Clearly, the mapping $x \mapsto \underline{d} f_x(0, 0)$ is measurable. Therefore, by the
Filippov Theorem (see, e.g. \cite[Thrm~8.2.10]{AubinFrankowska}) for any $u \in H^1(0, 1)$ there exists a measurable
selection $(a(x), v_1(x), v_2(x))$ of the map $x \mapsto \underline{d} f_x(0, 0)$ such that
$$
  \max_{(a, v) \in \underline{d} f_x(0, 0)} (a + v_1 u(x) + v_2 u'(x) \rangle) =
  a(x) + v_1(x) u(x) + v_2(x) u'(x)
$$
for a.e. $x \in (0, 1)$. Hence for any $u \in H^1(0, 1)$ one has
\begin{multline} \label{BolzaPr_IntermediateCodiff}
  \mathcal{I}(u_* + u) - \mathcal{I}(u_*) = u(0) - e^{-1} u(1) 
  + \max \Big( \int_0^1 \big( a(x) + v_1(x) u(x) + v_2(x) u'(x) \big) \, dx \Big) \\
  + \min \Big( \int_0^1 \big( b(x) + w_1(x) u(x) + w_2(x) u'(x) \big) \, dx \Big),
\end{multline}
where the maximum is taken over all measurable selections of the map $x \mapsto \underline{d} f_x(0, 0)$, and 
the minimum is taken over all measurable selections of the map $x \mapsto \overline{d} f_x(0, 0)$. 

Recall that $u \in H^1(0, 1)$ iff $u(x) = u(0) + \int_0^x \eta(s) ds$ for some $\eta \in L_2(0, 1)$ (see, e.g.
\cite{Leoni}). Therefore, instead of $\mathcal{I}(u)$ one can consider the functional 
$\mathcal{J} \colon \mathbb{R} \times L_2(0, 1) \to \mathbb{R}$ defined as
$\mathcal{J}(u_0, \eta) = \mathcal{I}(u)$, where $u(x) = u_0 + \int_0^x \eta(s) \, ds$. Denote 
$\eta_*(x) = \theta e^x$. Applying \eqref{BolzaPr_IntermediateCodiff} and integrating by parts one obtains that
\begin{align*}
  \mathcal{J}(\theta + u_0, \eta_* + \eta) - \mathcal{J}(\theta, \eta_*)
  &= \max_{(A, v_0, v) \in \underline{d} \mathcal{J}(\theta, \eta_*)} 
  \big( A + v_0 u_0 + \langle v, \eta \rangle \big) \\
  &+ \min_{(B, w_0, w) \in \overline{d} \mathcal{J}(\theta, \eta_*)} \big( B + w_0 u_0 + \langle w, \eta \rangle \big),
\end{align*}
where $\langle v, \eta \rangle = \int_0^1 v(x) \eta(x) \, dx$ is the inner product in $L_2(0, 1)$,
\begin{multline*}
  \underline{d} \mathcal{J}(\theta, \eta_*) 
  = \Big\{ (A, v_0, v) \in \mathbb{R} \times \mathbb{R} \times L_2(0, 1) \Bigm| A = \int_0^1 a(x) \, dx, \\ 
  v_0 = 1 - e^{-1} + \int_0^1 v_1(x) \, dx, \:
  v(x) = \int_x^1 v_1(s) \, ds + v_2(x) - e^{-1}, \: \\
  (a(\cdot), v_1(\cdot), v_2(\cdot)) \text{ is a measurable selection of the map } x \mapsto \underline{d} f_x(0, 0)
  \Big\},
\end{multline*}
and
\begin{multline*}
  \overline{d} \mathcal{J}(\theta, \eta_*) 
  = \Big\{ (B, w_0, w) \in \mathbb{R} \times \mathbb{R} \times L_2(0, 1) \Bigm| B = \int_0^1 b(x) \, dx, \\
  w_0 = \int_0^1 w_1(x) \, dx, \:
  w(x) = \int_x^1 w_1(s) \, ds + w_2(x), \\
  (b(\cdot), w_1(\cdot), w_2(\cdot)) \text{ is a measurable selection of the map } x \mapsto \overline{d} f_x(0, 0)
  \Big\}.
\end{multline*}
The sets $\underline{d} \mathcal{J}(\theta, \eta_*)$ and $\overline{d} \mathcal{J}(\theta, \eta_*)$ are obviously
convex. Let us verify that they are closed. For the sake of shortness, we consider only the set 
$\underline{d} \mathcal{J}(\theta, \eta_*)$. 

Note that the set $K$ of measurable selections of the map $x \mapsto \underline{d} f_x(0, 0)$ is obviously convex,
closed and bounded in $L_2(0, 1)$. Therefore it is weakly compact in $L_2(0, 1)$. It is easily seen that 
$\underline{d} \mathcal{J}(\theta, \eta_*)$ is the image of the set $K$ under a continuous map from the space 
$L_2(0, 1)$ endowed with the weak topology to the space $\mathbb{R} \times \mathbb{R} \times L_2(0, 1)$ endowed with the
weak topology as well. Hence the set $\underline{d} \mathcal{J}(\theta, \eta_*)$ is weakly compact, which implies that
it is closed in the norm topology due to the fact that this set is convex. 

Thus, the pair
$D \mathcal{J}(\theta, \eta_*) = [\underline{d} \mathcal{J}(\theta, \eta_*), \overline{d} \mathcal{J}(\theta, \eta_*)]$
is a global codifferential of $\mathcal{J}$ at the point $(\theta, \eta_*)$. Let us verify that this point is not
a global minimizer of $\mathcal{J}$ with the use of Theorem~\ref{Thrm_GlobOptCond_Alternative}.

\begin{remark}
It should be noted that a direct application of the global optimality conditions from Theorem~\ref{Thrm_GlobOptCond} to
problem \eqref{BolzaProblem} is very difficult, since it is unclear how to compute points $(a(z), v(z))$
defined in \eqref{GlobOptCond_WhenUseless} for this problem.
\end{remark}

The mapping $(b(x), w_1(x), w_2(x)) = (2 \theta e^x, 1, 0)$ is a measurable selection of the map
$x \mapsto \overline{d} f_x(0, 0)$ (see~\eqref{IntegrandGlobalHyperdiff}). Therefore, the point 
$z_* = (2\theta(e - 1), 1, w(\cdot))$ with $w(x) \equiv 1 - x$ belongs to $\overline{d} \mathcal{J}(\theta, \eta_*)$.
With the use of \eqref{IntegrandGlobalCodiff} and the Filippov Theorem one can easily check that any measurable
selection $(a(\cdot), v_1(\cdot), v_2(\cdot))$ of the map $x \mapsto \underline{d} f_x(0, 0)$ has the form
$$
  \begin{pmatrix} a(x) \\ v_1(x) \\ v_2(x) \end{pmatrix}
  = \begin{pmatrix} - 2 \theta (\alpha_2(x) + \alpha_4(x)) e^x \\ \alpha_1(x) - \alpha_2(x) 
  \\ \alpha_3(x) - \alpha_4(x) \end{pmatrix}
$$
for a.e. $x \in [0, 1]$ and for some $\alpha = (\alpha_1, \alpha_2, \alpha_3, \alpha_4) \in S_4$, where the set 
$S_4 \subset (L_2(0, 1))^4$ consists of all those $(\alpha_1, \alpha_2, \alpha_3, \alpha_4)$ for which all $\alpha_i$
are nonnegative and $\alpha_1(x) + \alpha_2(x) + \alpha_3(x) + \alpha_4(x) = 1$ for a.e. $x \in (0, 1)$. Therefore 
$[A, v_0, v] \in \underline{d} \mathcal{J}(\theta, \eta_*)$ iff there exists $\alpha \in S_4$ such that
\begin{gather*}
  A = - 2 \theta \int_0^1 e^x (\alpha_2(x) + \alpha_4(x)) \, dx, \quad
  v_0 = 1 - e^{-1} + \int_0^1 (\alpha_1(x) - \alpha_2(x)) \, dx, \\
  v(x) = - e^{-1} + \int_x^1 (\alpha_1(s) - \alpha_2(s)) \, ds + \alpha_3(x) - \alpha_4(x)
\end{gather*}
for a.e. $x \in [0, 1]$. Consequently, the point $(\xi, 0, 0)$ belongs to 
$\underline{d} \mathcal{J}(\theta, \eta_*) + z_*$ for some $\xi \in \mathbb{R}$ iff there exists $\alpha \in S_4$ such
that
\begin{align*}
  - 2 \theta \int_0^1 e^x (\alpha_2(x) + \alpha_4(x)) \, dx + 2\theta(e - 1) &= \xi, \\
  \int_0^1 (\alpha_1(x) - \alpha_2(x)) \, dx + 2 - e^{-1} &= 0, \\
  \int_x^1 (\alpha_1(s) - \alpha_2(s)) \, ds + \alpha_3(x) - \alpha_4(x) - e^{-1} + 1 - x &= 0
\end{align*}
for a.e. $x \in [0, 1]$. However, note that
$$
  \int_0^1 (\alpha_1(x) - \alpha_2(x)) \, dx \ge - \int_0^1 \alpha_2(x) \, dx \ge -1 > e^{-1} - 2
$$
due to the fact that $\alpha_1(x) \ge 0$ and $\alpha_2(x) \le 1$ for a.e. $x \in [0, 1]$. Thus,
the sets $\mathbb{R} \times \{ 0 \} \times \{ 0 \}$ and $\underline{d} \mathcal{J}(\theta, \eta_*) + z_*$ do not
intersect, which by Theorems~\ref{Thrm_GlobOptCond_Alternative} and \ref{Thrm_BoundednessBelow_GlobCodiff} implies that
the pair $(\theta, \eta_*(\cdot))$ with $\eta_*(x) = \theta e^x$ is not a point of global minimum of $\mathcal{J}$, and
this functional is unbounded below. Consequently, the function $u_*(x) = \theta e^x$ is not a global
minimizer of $\mathcal{I}$, and this functional is unbounded below as well.

\section{Conclusions}

In this article we obtained new necessary and sufficient global optimality conditions for DC optimization problems in
terms of global codifferentials. These optimality conditions are closely related to the method of codifferential
descent and are somewhat constructive, in the sense that they allow one to find ``global descent'' directions at 
non-optimal points. On the other hand, a direct usage of the global optimality conditions requires the knowledge of a
global codifferential of a DC function, and global codifferentials can be relatively easily computed (and manipulated
with) only in the piecewise affine case. Nevertheless, it seems possible to propose new methods for general DC
optimization problems utilising an approximation of global codifferential (cf. codifferential method \cite{BagUgon} and
aggregate codifferential method \cite{TorBagKar}). A development and analysis of such methods are interesting topics of
future research.

\bibliographystyle{abbrv}  
\bibliography{GlobOptimCond_bibl}

\end{document}